\documentclass[a4paper,11pt]{amsart}

\usepackage[all]{xy}

\usepackage[utf8]{inputenc}		% font and encoding
\usepackage[T1]{fontenc}
\usepackage[english]{babel}

\usepackage{amsfonts}			% ams-packages
\usepackage{amsmath}
\usepackage{amssymb}
\usepackage{amsthm}
\usepackage{mathtools}

\usepackage{hyperref}			% hyperlinks
	\hypersetup{colorlinks=false, urlcolor=black, linkcolor=black}

\usepackage{enumitem}			% arbitrary list items

\newcommand{\Z}{\mathbb{Z}}						% Integers
\newcommand{\R}{\mathbb{R}}						% Reals
\newcommand{\C}{\mathbb{C}}						% Complex numbers

\renewcommand{\S}{\mathbb{S}}					% Sphere
\newcommand{\B}{\mathbb{B}}						% Unit ball
\newcommand{\T}{\mathbb{T}}						% Torus

\newcommand{\Julia}{\mathcal{J}}				% Julia set
\newcommand{\Fatou}{\mathcal{F}}				% Fatou set

\newcommand{\dd}								% Differential d
	{\mathop{}\!\mathrm{d}}						
\newcommand{\ddn}[1]							% Powers of a differential d
	{\mathop{}\!\mathrm{d^{#1}}}

\newcommand{\abs}[1]							% Absolute value
	{\left| #1 \right|}
\newcommand{\smallabs}[1]						% Small absolute value bars which won't scale to the argument.
	{\lvert #1 \rvert}	
\newcommand{\norm}[1]							% Norm 
	{\left\lVert #1 \right\rVert}	
\newcommand{\smallnorm}[1]						% Small norm bars which won't scale to the argument.
	{\lVert #1 \rVert}						
\newcommand{\ip}[2]								% Inner product
	{\left< #1 , #2 \right>}

\DeclareMathOperator{\id}{id}					% Identity
\DeclareMathOperator{\vol}{vol}					% Volume (and the volume form)
\DeclareMathOperator{\spt}{spt}					% Support
\DeclareMathOperator*{\esssup}{ess\,sup}		% Essential supremum

\let\Re\relax									% Redefine real and imaginary part operators
\let\Im\relax
\DeclareMathOperator{\Re}{Re}
\DeclareMathOperator{\Im}{Im}					% Note: Imaginary part (pair is \Re).
\DeclareMathOperator{\im}{im}					% Note: Image (pair is \ker).

\newcommand{\push}[1]{{#1}_*\,}					% Pushforward (has better spacing in some contexts).

\newcommand{\hodge}{\mathtt{\star}\hspace{1pt}}
\newcommand{\cesob}{W^{d}_\text{CE}}
\newcommand{\cesobt}{\widetilde{W}^{d}_{\text{CE}}}

\newcommand{\cehom}[1]{H_{\text{CE}}^{#1}}
\newcommand{\cenorm}[1]{\norm{#1}_{d, \text{CE}}}

\newcommand{\fharm}[1]{\mathcal{H}_f^{#1}}

\newcommand{\cG}{\mathcal{G}}

\newcommand{\cE}{\mathcal{E}}

\newenvironment{acknowledgments}
	{\bigskip\noindent{\bf Acknowledgments.}}{}

\newtheorem{thm}{Theorem}[section]{\bf}{\it}
\newtheorem{lemma}[thm]{Lemma}
\newtheorem{prop}[thm]{Proposition}
\newtheorem{cor}[thm]{Corollary}
\newtheorem{qu}[thm]{Question}

\newtheorem{innercustomthm}{Theorem}[section]{\bf}{\it}
\newenvironment{customthm}[1]
	{\renewcommand\theinnercustomthm{#1}\innercustomthm}
	{\endinnercustomthm}

\theoremstyle{definition}

\theoremstyle{remark}
\newtheorem{rem}[thm]{Remark}

\numberwithin{equation}{section}

\begin{document}

\title[]{Sharp cohomological bound for uniformly quasiregularly elliptic manifolds}
\author{Ilmari Kangasniemi}
\address{Department of Mathematics and Statistics, P.O. Box 68 (Gustaf H\"allstr\"omin katu 2b), FI-00014 University of Helsinki, Finland}
\email{ilmari.kangasniemi@helsinki.fi}

\begin{abstract}
	We show that if a compact, connected, and oriented $n$-manifold $M$ without boundary admits a non-constant non-injective uniformly quasiregular self-map, then the dimension of the real singular cohomology ring $H^*(M; \R)$ of $M$ is bounded from above by $2^n$. This is a positive answer to a dynamical counterpart of the Bonk-Heinonen conjecture on the cohomology bound for quasiregularly elliptic manifolds. The proof is based on an intermediary result that, if $M$ is not a rational homology sphere, then each such uniformly quasiregular self-map on $M$ has a Julia set of positive Lebesgue measure.
\end{abstract}

\thanks{This work was supported by the doctoral program DOMAST of the University of Helsinki and the Academy of Finland project \#297258.}
\date{\today}

\maketitle 

\section{Introduction}\label{sect:intro}

A continuous map $f \colon M \to N$ between two oriented Riemannian $n$-manifolds ($n \geq 2$) is \emph{$K$-quasiregular} for $K\geq 1$ if it belongs to the Sobolev space $W^{1, n}(M, N)$ and satisfies the distortion inequality $\norm{Df}^n \leq K J_f$ almost everywhere on $M$ in the Lebesgue sense. Here, $\norm{Df}$ is the operator norm of $Df$, and $J_f$ is the Jacobian determinant $\det Df$ of $f$. A quasiregular self-map $f \colon M \to M$ is called \emph{uniformly $K$-quasiregular} if all iterates $f^m, m\geq 1$, are $K$-quasiregular.

Our main result is that a closed manifold admitting a non-constant non-injective uniformly quasiregular self-map has uniformly bounded cohomology. Here and in what follows, we call a Riemannian manifold $M$ \emph{closed} if it is compact, connected, oriented and without boundary.

\begin{thm}\label{thm:dynamical_bonk_heinonen}
	Let $n \geq 2$ and let $M$ be a compact, connected, and oriented Riemannian $n$-manifold without boundary. Suppose that $M$ admits a non-constant non-injective uniformly quasiregular self-map $f \colon M \to M$. Then, for every $k \in \{0, \ldots, n\}$,
	\[
		\dim H^k(M; \R) \leq \binom{n}{k}.
	\]
	In particular,
	\[
		\dim H^*(M; \R) \leq 2^n. 
	\]
\end{thm}

Note that the $n$-torus $\T^n$ satisfies $\dim H^k(\T^n; \R) = \binom{n}{k}$ for $k = 0, \ldots, n$, and that $\T^n$ admits a non-constant non-injective uniformly quasiregular self-map for every $n \geq 2$. Hence, the cohomology bound in Theorem \ref{thm:dynamical_bonk_heinonen} is sharp.

Theorem \ref{thm:dynamical_bonk_heinonen} is a dynamical counterpart of a conjecture of Bonk and Heinonen regarding the cohomological bound for quasiregularly elliptic manifolds. Recall that a closed $n$-manifold $M$ is called \emph{$K$-quasiregularly elliptic} for $K \geq 1$ if there exists a non-constant $K$-quasiregular map $\R^n \to M$. Correspondingly, we call a manifold $M$ \emph{uniformly quasiregularly elliptic} if it admits a non-constant non-injective uniformly quasiregular self-map. As a consequence of Zalcman's lemma, uniform quasiregular ellipticity implies quasiregular ellipticity for closed manifolds $M$; see e.g.\ Kangaslampi \cite[Theorem 5.7]{Kangaslampi-thesis}. The converse whether all closed quasiregularly elliptic $n$-manifolds are uniformly quasiregularly elliptic is true for $n = 2, 3$, but remains open in dimensions $n > 3$; see e.g.\ \cite[Theorem 7.1]{Kangaslampi-thesis} and \cite[p.220]{Bonk-Heinonen_Acta}. 
 
By a theorem of Bonk and Heinonen \cite[Theorem 1.1]{Bonk-Heinonen_Acta}, if $M$ is a closed $K$-quasiregularly elliptic $n$-manifold, then $\dim H^*(M; \R) \leq C$, where $C = C(n, K)$ is a constant depending on $n$ and $K$. In \cite[p. 222]{Bonk-Heinonen_Acta}, Bonk and Heinonen conjecture that the bound is independent of the distortion constant $K$, and more precisely, that the optimal bound is $\dim H^*(M; \R) \leq 2^n$.

This conjecture holds in dimensions $2$ and $3$; see \cite[Corollary 1.6]{Bonk-Heinonen_Acta}. It is also known that if $M$ is a quasiregularly elliptic manifold with a fundamental group $\pi_1(M)$ of polynomial order $n$, the conjectured bound of $2^n$ holds; see \cite[Corollary 1.4]{LuistoPankka2016Paper}. However, the conjecture is still open for general quasiregularly elliptic manifolds of dimension $n \geq 4$. Note also that without the compactness assumption, there exists no cohomological bound dependent only on $n$; see e.g.\ the Picard-type constructions of Rickman \cite{Rickman1985paper} and Drasin--Pankka \cite{DrasinPankka2015paper} of quasiregular maps from $\R^n$ onto punctured spheres.

The conjecture of Bonk and Heinonen is related to a question of Gromov and Rickman on whether all simply connected closed manifolds are quasiregularly elliptic; see \cite[p.183]{Rickman1988proc}, \cite[p. 200]{Gromov1981proc} and \cite[p.63, p.67]{Gromov2007book}. Special attention was given to the specific case of whether $(S^2 \times S^2) \# (S^2 \times S^2)$ is quasiregularly elliptic, which was eventually given an affirmative answer by Rickman \cite{Rickman2006paper}. Due to this, the question of whether $(S^2 \times S^2) \# (S^2 \times S^2)$ is uniformly quasiregularly elliptic is of particular interest.

\subsection{Positive measure of Julia sets}

One of the key ingredients in the proof of Theorem \ref{thm:dynamical_bonk_heinonen} is a result that for a closed manifold $M$ which is not a rational cohomology sphere, all uniformly quasiregular self-maps on $M$ have large Julia sets. Recall that for a uniformly quasiregular self-map $f \colon M \to M$ on a closed $n$-manifold $M$, the \emph{Fatou set $\Fatou_f$ of $f$} is the union of all open sets $V \subset M$ on which the family $\{f^k|V \colon k \geq 1\}$ is normal, and the \emph{Julia set $\Julia_f$ of $f$} is the complement $M \setminus \Fatou_f$. For a more detailed exposition, see \cite{HinkkanenMartinMayer2004paper}.

Suppose that the uniformly quasiregular map $f$ on $M$ is non-constant and non-injective. In \cite{OkuyamaPankka2014paper}, Okuyama and Pankka construct an $f$-invariant probability measure $\mu_f$ on $M$. The measure $\mu_f$ is ergodic and balanced under $f$, and satisfies $\spt \mu_f = \Julia_f$. It turns out that, if $M$ is not a rational homology sphere, then $\mu_f$ is absolutely continuous with respect to the Lebesgue measure on $M$.

\begin{thm}\label{thm:positive_julia_set}
	Let $n \geq 2$, and let $M$ be a compact, connected, and oriented Riemannian $n$-manifold without boundary. Suppose that $M$ admits a non-constant non-injective uniformly quasiregular self-map $f \colon M \to M$. If $M$ is not a rational homology sphere, then the measure $\mu_f$ is absolutely continuous with respect to the Lebesgue measure $m_n$ on $M$, and $m_n(\Julia_f) > 0$.
\end{thm}

Note that there exist non-constant non-injective uniformly quasiregular maps with Julia sets of Lebesgue measure zero on the $n$-sphere $\S^n$; see e.g. \cite[Theorem 2]{Mayer1997paper} or \cite[Section 6.2]{Astola-Kangaslampi-Peltonen}.

In \cite[Conjecture 1.4]{MartinMayer2003paper}, Martin and Mayer conjecture that for $n \geq 3$, every uniformly quasiregular map $f \colon \S^n \to \S^n$ with a Julia set of positive measure is of the Latt\'es type. Recall that a uniformly quasiregular map $f \colon M \to M$ is of the \emph{Latt\'es type} if there exists a discrete group $\Gamma$ of isometries of $\R^n$, a quasiregular map $\varphi \colon \R^n \to M$ which is automorphic with respect to $\Gamma$ in the strong sense, and a linear conformal map $A \colon \R^n \to \R^n$ satisfying $A\Gamma A^{-1} \subset \Gamma$ and $f \circ \varphi = \varphi \circ A$; we refer to \cite[Section 21.4]{IwaniecMartin2001book} or \cite[Definition 2.2 and Theorem 2.3]{Astola-Kangaslampi-Peltonen} for a more detailed exposition. The conjecture of Martin and Mayer is a uniformly quasiregular version of the no invariant lines field conjecture of holomorphic dynamics. In light of Theorem \ref{thm:positive_julia_set}, this question takes the following form on closed manifolds with nontrivial rational cohomology.

\begin{qu}\label{qu:is_everything_lattes?}
	Let $M$ be a compact connected oriented Riemannian manifold without boundary. Suppose that $M$ is not a rational homology sphere. Is every non-constant non-injective uniformly quasiregular self-map on $M$ of the Latt\'es type?
\end{qu}

In \cite[Theorem 1.3]{MartinMayer2003paper}, Martin and Mayer prove a weaker statement that a uniformly quasiregular map $f \colon \S^n \to \S^n$ is of the Latt\'es type if it has a positive measured \emph{set of conical points}. Recall that a point $x_0 \in \S^n$ is a conical point of a uniformly quasiregular map $f \colon \S^n \to \S^n$ if there exist sequences $\rho_j \to 0$ and $k_j \to \infty$ for which $f^{k_j}(x_0 + \rho_j x)$ converges uniformly to a non-constant quasiregular map $\psi \colon \B^n \to \S^n$.

\subsection{Sketch of the proof of Theorem \ref{thm:dynamical_bonk_heinonen}}

We now outline the key ideas of the proof of the main results.

Let the mapping $f$ and the manifold $M$ be as in Theorem \ref{thm:dynamical_bonk_heinonen}. Due to a previous joint work with Pankka \cite{KangasniemiPankka2017paper}, the manifold $M$ admits a Sobolev-de Rham cohomology $\cehom{*}(M)$, which is naturally isomorphic to $H^*(M; \R)$, and for which a quasiregular map $M \to M$ induces a natural pull-back map $\cehom{*}(M) \to \cehom{*}(M)$. We consider a corresponding cohomology with complex coefficients $\cehom{*}(M; \C)$, which is naturally isomorphic to $\cehom{*}(M) \otimes \C$.  

By the invariant conformal structure of Iwaniec--Martin, see \cite[Theorem 5.1]{IwaniecMartin1996paper} or \cite[Theorem 21.5.1]{IwaniecMartin2001book}, there is a measurable Riemannian metric $\ip{\cdot}{\cdot}_f$ on $M$ satisfying
\[
	\ip{f^*\alpha}{f^*\beta}_f = \left(\ip{\alpha}{\beta}_f \circ f\right) J_f^\frac{2k}{n}
\]
almost everywhere on $M$ for all measurable $k$-forms $\alpha, \beta$ on $M$. Consequently, we obtain measurable pointwise norms $\abs{\cdot}_f$ for $k$-covectors, and $L^p$-norms $\norm{\cdot}_{f, p}$ for $1 \leq p \leq \infty$.

For $0 < k < n$, every complex cohomology class $c \in \cehom{k}(M; \C) $ contains a unique measurable complex $k$-form $\omega_c$ which minimizes the norm $\norm{\cdot}_{f, n/k}$. By a computation involving the quasiregular pushforward operator as in \cite{KangasniemiPankka2017paper}, we obtain 
\[
	f^*\omega_c = \omega_{f^*c} 
\]
for all $c \in \cehom{k}(M; \C) $. Furthermore, 
\[
	\omega_{\lambda c} = \lambda \omega_c 
\] 
whenever $c \in \cehom{k}(M; \C) $ and $\lambda \in \C$. This allows us to associate eigenvectors $c$ of $f^*$ in complex cohomology with eigenvectors $\omega_c$ of $f^*$ on the level of measurable complex differential forms.

The eigenvector form $\omega_c$ now yields a representation for $\mu_f$, which proves Theorem \ref{thm:positive_julia_set}. More precisely, we have the following. 

\begin{prop}\label{prop:invariant_measure_representation}
	Let $f \colon M \to M$ be a non-constant uniformly quasiregular map of degree at least two on a compact, connected, and oriented Riemannian $n$-manifold $M$ without boundary, where $n \geq 2$. Let $0 < k < n$ and suppose $c \in \cehom{k}(M; \C) \setminus \{0\}$ is a complex cohomology class satisfying $f^*c = \lambda c$ for some $\lambda \in \C \setminus \{0\}$. Then the invariant probability measure $\mu_f$ of Okuyama and Pankka has the representation
	\[
	\mu_f = \frac{\abs{\omega_c}_f^{\frac{n}{k}}}{\norm{\omega_c}_{f, \frac{n}{k}}^\frac{n}{k}} \vol_M
	\]
	as a measurable $n$-form, where $\omega_c \in c$ is the element minimizing the norm $\norm{\cdot}_{f, n/k}$ in $c$.
\end{prop}

Note that if $c_1$ and $c_2$ are two different eigenvectors of $f^* \colon \cehom{k }(M; \C) \to \cehom{k }(M; \C)$, then by Proposition \ref{prop:invariant_measure_representation} the corresponding $k$-forms $\omega_{c_1}$ and $\omega_{c_2}$ have the same support $\Julia_f$, and there exists a constant $C_{12}$ for which $\abs{\omega_{c_1}}_f = C_{12} \abs{\omega_{c_2}}_f$ almost everywhere. Using similar methods we show that the complex angle between $\omega_{c_1}$ and $\omega_{c_2}$ in the inner product $\ip{\cdot}{\cdot}_f$ is constant on $\Julia_f$. We formulate the result as follows. 

\begin{prop}\label{prop:norm_mimimizer_angles}
	Let $f \colon M \to M$ be a non-constant uniformly quasiregular map of degree at least two on a compact, connected, and oriented Riemannian $n$-manifold $M$ without boundary, where $n \geq 2$. Let $0 < k < n$, suppose $c_1, c_2 \in \cehom{k}(M; \C) \setminus \{0\}$ are complex cohomology classes satisfying $f^*c_1 = \lambda_1 c_1$ and $f^*c_2 = \lambda_2 c_2$ for some $\lambda_1, \lambda_2 \in \C \setminus \{0\}$, and let $\omega_{c_1} \in c_1$ and $\omega_{c_2} \in c_2$ minimize the norm $\norm{\cdot}_{f, n/k}$ in $c_1$ and $c_2$, respectively. Then the point-wise complex $f$-angle element
	\[
	x \mapsto \frac{\ip{(\omega_{c_1})_x}{(\omega_{c_2})_x}_f}{\abs{(\omega_{c_1})_x}_f \abs{(\omega_{c_2})_x}_f}
	\]
	between $\omega_{c_1}$ and $\omega_{c_2}$ is constant almost everywhere on $\Julia_f$. Furthermore, if $\lambda_1 \neq \lambda_2$, then $\ip{(\omega_{c_1})_x}{(\omega_{c_2})_x}_f = 0$ at almost every $x \in M$.
\end{prop}

By Propositions \ref{prop:invariant_measure_representation} and \ref{prop:norm_mimimizer_angles}, we may select $D = \dim \cehom{k}(M)$ measurable complex $k$-forms $\omega_i$ which are almost everywhere pairwise orthogonal in the complex Riemannian metric $\ip{\cdot}{\cdot}_f$ and supported on the Julia set $\Julia_f$ of positive measure. Hence, by observing the cotangent bundle $\wedge^k T^*M$ at a suitable point $x \in \Julia_f$, we obtain the bound
\[
D \leq \dim_\C ((\wedge^k T_x^*M) \otimes \C) = \binom{n}{k},
\]
thus proving Theorem \ref{thm:dynamical_bonk_heinonen}.

\medskip
This article is organized as follows. In Section \ref{sect:preliminaries}, we recall the $L^p$- and Sobolev spaces of differential forms, with special emphasis on forms with complex coefficients. In Section \ref{sect:conf_cohomology}, we discuss conformal cohomology and the quasiregular push-forward, based on the exposition in \cite{KangasniemiPankka2017paper}. In Section \ref{sect:cohomology_representation}, we recall the invariant conformal structure of Iwaniec and Martin, and apply it to obtain the desired cohomology representation by norm-minimizing forms. In Section \ref{sect:inv_equilibrium_measure}, we discuss the necessary results related to the invariant equilibrium measure of Okuyama and Pankka. Finally, in Section \ref{sect:main_results}, we prove Theorems \ref{thm:dynamical_bonk_heinonen} and \ref{thm:positive_julia_set}.

\begin{acknowledgments}
	The author thanks his thesis advisor Pekka Pankka for helpful critical comments on this paper.
\end{acknowledgments}

%%%%%%%%%%%%%%%%%%%%%%%%%%%%%%%%%%%%%%%%%%%%%%%%%%%%%%%%%%%%%%%%%%%%%%%%%%%%%%%%%%%%%%%%%%%%%%%%%%
%%%%%%%%%%%%%%%%%%%%%%%%%%%%%%%%%%%%%%%%%%%%%%%%%%%%%%%%%%%%%%%%%%%%%%%%%%%%%%%%%%%%%%%%%%%%%%%%%%

\section{Preliminaries}\label{sect:preliminaries}

In this section, we recall the Lebesgue spaces $L^p$ and the partial Sobolev spaces $W^{d, p, q}$ of differential forms with real or complex coefficients. For further information on the real versions of these spaces, see e.g. Iwaniec--Scott--Stroffolini \cite{IwaniecScottStroffolini1999paper} and Iwaniec--Lutoborski \cite{IwaniecLutoborski1993paper}.

\subsection{Real and complex $L^p$-spaces of differential forms.} \label{subsect:diff_forms}

Throughout this paper $M$ stands for a closed $n$-manifold with the Riemannian metric $\ip{\cdot}{\cdot}$, $n \geq 2$. We extend the notation $\ip{\cdot}{\cdot}$ to the induced Riemannian metric on coexterior bundles $\wedge^k T^*M$, that is, given $x \in M$, the map $\ip{\cdot}{\cdot} \colon (\wedge^k T_x^*M) \times (\wedge^k T_x^*M) \to \R$ is the Grassmann inner product defined by
\begin{equation}\label{eq:grassmann}
	\langle \sigma(v_1)\wedge \cdots \wedge \sigma(v_k), \sigma(w_1)\wedge \cdots \wedge \sigma(w_k)\rangle 
	= \det \left(\langle v_i,w_j \rangle\right)_{ij},
\end{equation}
where $v_1,\ldots, v_k,w_1,\ldots, w_k\in T_x M$ and $\sigma \colon TM \to T^*M$ is the natural bundle map given by $\sigma(v)(w) = \ip{v}{w}$ for $v,w\in T_xM$. The Riemannian metric induces a point-wise norm $\abs{\cdot} \colon \wedge^k T^*M \to [0, \infty)$ on $\wedge^k T^*M$ by $\abs{\omega} = \ip{\omega}{\omega}^{1/2}$.

Let $U \subset M$ be a domain in $M$, with the possibility that $U = M$. We denote the set of \emph{measurable $k$-forms} on $U$, that is, measurable sections $U \to \wedge^k T^*M$, by $\Gamma(\wedge^k U)$. 
In addition, let $\Gamma(\wedge^k U; \C)$ be the space of \emph{measurable complex $k$-forms}, consisting of all measurable sections $U \to \wedge^k T^*M \otimes \C$. An element $\omega \in \Gamma(\wedge^k U; \C)$ is of the form $\omega = \alpha + i\beta$, where $\alpha, \beta \in \Gamma(\wedge^k U)$. Consequently, $\Gamma(\wedge^k U; \C)$ is naturally isomorphic to $\Gamma(\wedge^k U) \otimes \C$. The point-wise inner product induced by the Riemannian metric on $\Gamma(\wedge^k U)$ extends to a point-wise complex inner product on $\Gamma(\wedge^k U; \C)$ in the usual way, by
\begin{equation}\label{eq:complex_inner_product}
\ip{\alpha_x + i\beta_x}{\alpha_x'+i\beta_x'} = \ip{\alpha_x}{\alpha_x'} + \ip{\beta_x}{\beta_x'} + i\ip{\beta_x}{\alpha_x'} - i\ip{\alpha_x}{\beta_x'} 
\end{equation}
for $x \in U$ and $\alpha_x, \beta_x, \alpha_x', \beta_x' \in \wedge^k T_x^*M$. Consequently, the point-wise norm $\abs{\cdot}$ also extends to $\Gamma(\wedge^k U; \C)$.

Given $p \in [1, \infty)$, the space $L^p(\wedge^k U)$ of \emph{$p$-integrable $k$-forms} is the set of all measurable $k$-forms $\omega \in \Gamma(\wedge^k U)$ for which the $L^p$-norm
\[
	\norm{\omega}_p = \left(\int_M |\omega(x)|^p \vol_M\right)^{1/p}
\]
is finite; here $\vol_M$ stands for the \emph{volume form} of $M$ induced by the Riemannian metric and the chosen orientation. The space $L^p(\wedge^k U; \C)$ of \emph{$p$-integrable complex $k$-forms} is similarly defined as the space of all $\omega \in \Gamma(\wedge^k U; \C)$ for which $\norm{\omega}_p < \infty$. Note that, given $\omega = \alpha + i\beta \in L(\wedge^k U; \C)$, we have the elementary estimate
\begin{equation}\label{eq:complex_lp_estimate_finite_p}
	2^{\frac{1}{p} - \frac{3}{2}} \left( \norm{\alpha}_p + \norm{\beta}_p \right)  
		\leq \norm{\omega}_p \leq \norm{\alpha}_p + \norm{\beta}_p.
\end{equation}
Hence, we may equivalently define the space $L^p(\wedge^k U; \C)$ by $L^p(\wedge^k U; \C) = L^p(\wedge^k U) \otimes \C$.

The space $L^\infty(\wedge^k U)$ of \emph{essentially bounded $k$-forms} is, as usual, the set of those forms $\omega \in \Gamma(\wedge^k U)$ for which the norm
\[
	\norm{\omega}_\infty = \esssup_{x\in M} |\omega_x|
\]
is finite. The complex counterpart $L^\infty(\wedge^k U; \C)$ is the space of $\omega \in \Gamma(\wedge^k U; \C)$ satisfying $\norm{\omega}_\infty < \infty$. As in the case of $1 \leq p < \infty$, the estimate
\[
	2^{-1} \left( \norm{\alpha}_\infty + \norm{\beta}_\infty \right) 
	\leq \norm{\omega}_\infty \leq \norm{\alpha}_\infty + \norm{\beta}_\infty
\]
for $\omega = \alpha + i\beta \in L^\infty(\wedge^k U; \C)$ shows that $L^\infty(\wedge^k U; \C) = L^\infty(\wedge^k U) \otimes \C$.

\subsection{Partial Sobolev spaces of differential forms}

The \emph{Hodge star} is the bundle isometry $\hodge \colon \wedge^k T^*M \to \wedge^{n-k} T^*M$ defined by
\begin{equation}\label{eq:hodge_star}
	\tau \wedge \hodge \omega = \ip{\tau}{\omega} \vol_M(x)
\end{equation}
for $x \in M$ and $\tau, \omega \in \wedge^k T^*_x M$. We may also use equation \eqref{eq:hodge_star} to define the Hodge star $\hodge \colon \wedge^k T^*M \otimes \C \to \wedge^{n-k} T^*M \otimes \C$ in the case of complex coefficients; here, the wedge product on $\wedge^* T^*M \otimes \C$ is defined with a standard bilinear extension. By a simple verification, the complex Hodge star follows the formula 
\[
	\hodge(\alpha + i\beta) = (\hodge \alpha) -i(\hodge \beta)
\]
for $x \in M$ and $\alpha, \beta \in \wedge^k T^*_x M$.

Let $U \subset M$ again be a domain in $M$. We denote $C^\infty(\wedge^k U)$ the space of \emph{smooth $k$-forms}, and define its complex counterpart by $C^\infty(\wedge^k U; \C) = C^\infty(\wedge^k M) \otimes \C$. The \emph{exterior derivative} $d \colon C^\infty(\wedge^k U) \to C^\infty(\wedge^{k+1} U)$ extends to $C^\infty(\wedge^k U; \C)$ by 
\[
	d(\alpha + i\beta) = d\alpha +id\beta
\]
for $\alpha, \beta \in C^\infty(\wedge^k M)$. The \emph{coexterior derivative} is defined in both $C^\infty(\wedge^k U)$ and $C^\infty(\wedge^k U; \C)$ by $d^* = (-1)^{nk+n+1} \hodge d \hodge$, and the complex version again follows the formula
\[
	d^*(\alpha + i\beta) = d^*\alpha +id^*\beta
\]
where $\alpha, \beta \in C^\infty(\wedge^k U)$. We also define the spaces of \emph{compactly supported smooth $k$-forms} $C^\infty_0(\wedge^k U)$ and $C^\infty_0(\wedge^k U; \C) = C^\infty_0(\wedge^k U) \otimes \C$ in the usual manner.

Let $\omega \in L^1(\wedge^k U)$. A \emph{weak exterior derivative} of $\omega$ is a measurable $(k+1)$-form $d\omega$ satisfying
\begin{equation}\label{eq:weak_ext_derivative_definition}
	\int_U \ip{d\omega}{\eta} \vol_M = \int_U \ip{\omega}{d^*\eta} \vol_M
\end{equation}
for all $\eta \in C^\infty_0(\wedge^k U)$. We denote the space of $k$-forms with a weak exterior derivative by $W^d(\wedge^k U)$. The complex counterpart $W^d(\wedge^k U; \C)$ is defined analogously by using test forms in $C^\infty_0(\wedge^k U; \C)$. 

If $\alpha + i\beta \in W^d(\wedge^k U) \otimes \C$, it is easily verified that $d\alpha + id\beta$ is a weak exterior derivative of $\alpha + i\beta$, and hence $W^d(\wedge^k U) \otimes \C \subset W^d(\wedge^k U; \C)$. The converse inclusion $W^d(\wedge^k U; \C) \subset W^d(\wedge^k U) \otimes \C$ follows by using test forms $\eta \in C^\infty(\wedge^k U; \C)$ with zero imaginary part. Hence, $W^d(\wedge^k U; \C) = W^d(\wedge^k U) \otimes \C$. 

The \emph{partial Sobolev $(p, q)$-space} $W^{d, p, q}(\wedge^k U)$ is the space of all $k$-forms $\omega \in L^p(\wedge^k U) \cap W^{d}(\wedge^k U)$ for which $d\omega \in L^q(\wedge^k U)$. The complex counterpart $W^{d, p, q}(\wedge^k U; \C)$ is defined as usual, and again satisfies the identity $W^{d, p, q}(\wedge^k U; \C) = W^{d, p, q}(\wedge^k U) \otimes \C$. The norm $\norm{\cdot}_{p, q}$ on the spaces $W^{d, p, q}(\wedge^k U)$ and $W^{d, p, q}(\wedge^k U; \C)$ is defined by
\[
	\norm{\omega}_{p, q} = \norm{\omega}_p + \norm{d\omega}_q
\]
for $\omega \in W^{d, p, q}(\wedge^k U)$ or $\omega \in W^{d, p, q}(\wedge^k U; \C)$. The spaces $W^{d, p, p}(\wedge^k U)$ and $W^{d, p, p}(\wedge^k U; \C)$ are also denoted $W^{d, p}(\wedge^k U)$ and $W^{d, p}(\wedge^k U; \C)$, respectively. 

We define partial Sobolev $(p, q)$-spaces $W^{d^*, p, q}(\wedge^k U)$ and $W^{d^*, p, q}(\wedge^k U; \C)$ for the \emph{weak coexterior derivative} $d^*$ analogously. A weak coexterior derivative of a $k$-form $\omega$ in $L^1(\wedge^k U)$ or $L^1(\wedge^k U; \C)$ is a measurable $(k-1)$-form $d^*\omega$ satisfying
\begin{equation}\label{eq:weak_coext_derivative_definition}
	\int_U \ip{d^*\omega}{\eta} \vol_M = \int_U \ip{\omega}{d\eta} \vol_M
\end{equation}
for every $\eta$ in $C^\infty_0(\wedge^k M)$ or $C^\infty_0(\wedge^k M; \C)$, respectively. We also use the notations $W^{d^*, p}$ and $W^{d^*}$, which are defined in the obvious manner.

Finally, we remark on two classical results related to Sobolev differential forms. The first of these is the \emph{Sobolev-Poincar\'e inequality}. Let $\Omega$ be either a closed $n$-manifold or a cube in $\R^n$, let $q \in (1, \infty)$, and let $\omega \in  dW^{d, 1, q}(\wedge^{k-1} \Omega)$. Then there exists a form $\tau \in W^{d, 1, q}(\wedge^{k-1} \Omega)$ for which $d\tau = \omega$ and for all $p \in (1, \infty)$ which satisfy
\begin{equation}\label{eq:sobolev_poincare_condition}
	\frac{1}{q} - \frac{1}{p} \leq \frac{1}{n},
\end{equation}
we have $\tau \in L^p(\wedge^{k-1} \Omega)$ and
\begin{equation}\label{eq:sobolev_poincare_ineq}
	\norm{\tau}_p \leq C\norm{\omega}_q.
\end{equation}
If $\Omega$ is a closed manifold $M$, the constant $C$ depends on $p$, $q$, and the manifold $M$. However, if $\Omega$ is a cube in $\R^n$, the constant $C$ depends only on $n$, $p$ and $q$, and is therefore independent of the cube itself. For details, see e.g.\ Iwaniec--Lutoborski \cite[Corollary 4.1 and Corollary 4.2]{IwaniecLutoborski1993paper} for the formulation on cubes, and Gol'dshtein--Troyanov \cite{GoldsteinTroyanov2006paper} for the formulation on closed manifolds.

As a simple corollary for this, a similar inequality holds for the weak coexterior derivative; if $q \in (1, \infty)$ and $\omega \in d^*W^{d^*,1, q}(\wedge^{k+1} \Omega)$, then there exists $\tau \in W^{d^*,1, q}(\wedge^{k+1} \Omega)$ for which $d^*\tau = \omega$ and \eqref{eq:sobolev_poincare_ineq} holds for all $p \in (1, \infty)$ satisfying \eqref{eq:sobolev_poincare_condition}. We also note here that \eqref{eq:sobolev_poincare_ineq} generalizes to complex $\omega$ and $\tau$ by seperate application on real and imaginary parts followed by use of \eqref{eq:complex_lp_estimate_finite_p}.

The other result is the \emph{Hodge decomposition} of measurable differential forms. Given $1 < p < \infty$ and $0 \leq k \leq n$, every $p$-integrable $k$-form $\omega \in L^p(\wedge^k U)$ can be written in the form
\begin{equation}\label{eq:hodge_decomposition}
	\omega = d\alpha + d^*\beta + \gamma,
\end{equation}
where $\alpha \in W^{d, p}(\wedge^{k-1} U)$, $\beta \in W^{d^{*}, p}(\wedge^{k+1} U)$, and $\gamma \in \ker d \cap \ker d^*$.
If the domain $U$ is the entire closed manifold $M$, the forms $d\alpha$, $d^*\beta$ and $\gamma$ are unique, but in the general case uniqueness requires additional boundary conditions on $\alpha$, $\beta$ and $\gamma$; see Iwaniec--Scott--Stroffolini \cite[Theorem 5.7]{IwaniecScottStroffolini1999paper} or Schwarz \cite[Theorems 2.4.2, 2.4.8, 2.4.14]{Schwarz1995Book}. 

We note that if $U = M$, the harmonic part $\gamma$ of decomposition \eqref{eq:hodge_decomposition} is smooth by a version of Weyl's lemma; see e.g.\ Warner \cite[Theorem 6.5]{Warner1983book}. In the general case of domains $U$ in $M$, $\gamma$ is actually smooth on the closure $\overline{U}$ of $U$ if the boundary conditions for $\alpha$, $\beta$, and $\gamma$ are selected appropriately; see e.g.\ \cite[Theorems 2.2.6, 2.2.7]{Schwarz1995Book}.

The decomposition is a generalization of the classical Hodge decomposition for smooth forms, which states that a smooth $k$-form $\omega \in C^\infty(\wedge^k U)$ decomposes in the manner of \eqref{eq:hodge_decomposition} with smooth $\alpha$, $\beta$ and $\gamma$. The smooth Hodge decomposition is again unique on the entire closed manifold $M$, and can be made unique on domains $U \subset M$ with boundary conditions; see e.g.\ \cite[Section 5.1]{IwaniecScottStroffolini1999paper}. In both the measurable and smooth cases, the decomposition \eqref{eq:hodge_decomposition} generalizes to forms with complex coefficients by taking decompositions of the real and imaginary parts separately.

%%%%%%%%%%%%%%%%%%%%%%%%%%%%%%%%%%%%%%%%%%%%%%%%%%%%%%%%%%%%%%%%%%%%%%%%%%%%%%%%%%%%%%%%%%%%%%%%%%
%%%%%%%%%%%%%%%%%%%%%%%%%%%%%%%%%%%%%%%%%%%%%%%%%%%%%%%%%%%%%%%%%%%%%%%%%%%%%%%%%%%%%%%%%%%%%%%%%%

\section{Conformal cohomology and quasiregular maps}\label{sect:conf_cohomology}

The proofs of Theorems \ref{thm:dynamical_bonk_heinonen} and \ref{thm:positive_julia_set} rely heavily on the tools developed in a joint paper with Pankka \cite{KangasniemiPankka2017paper}. In this section, we recall the conformal Sobolev cohomology and quasiregular push-forward. We focus only on the results, and refer to \cite{KangasniemiPankka2017paper} for proofs.

\subsection{Conformal Sobolev cohomology}

Given $1 \leq p \leq \infty$ and $0 \leq k \leq n$, we define the \emph{sharp and flat $L^p$-spaces of differential forms} $L^{p, \sharp}(\wedge^k M)$ and $L^{p, \flat}(\wedge^k M)$ by
\[
	L^{p, \sharp}(\wedge^k M) = \bigcup_{s \in (p, \infty]} L^{s}(\wedge^k M)
\]
and
\[
	L^{p, \flat}(\wedge^k M) = \bigcap_{s \in [1, p)} L^{s}(\wedge^k M).
\]
The complex counterparts $L^{p, \sharp}(\wedge^k M; \C)$ and $L^{p, \flat}(\wedge^k M; \C)$ are defined correspondingly, and satisfy $L^{p, \sharp}(\wedge^k M; \C) = L^{p, \sharp}(\wedge^k M) \otimes \C$ and $L^{p, \flat}(\wedge^k M; \C) = L^{p, \flat}(\wedge^k M) \otimes \C$.

We define the \emph{Sobolev spaces of conformal exponent} $\cesob(\wedge^k M)$ by
\[
	\cesob(\wedge^k M) = \begin{cases}
		W^{d, n, \infty}(\wedge^0 M), & k=0, \\
		W^{d, \frac{n}{k}, \frac{n}{k+1}}(\wedge^k M), & 1 \leq k \leq n-1, \\
		L^1(\wedge^n M), & k = n, \\
		0, & k>n,
	\end{cases}
\]
and the norm $\cenorm{\cdot}$ for $\omega \in \cesob(\wedge^k M)$ by 
\[
	\cenorm{\omega} = \begin{cases}
		\norm{\omega}_{\infty, n}, & k=0, \\
		\norm{\omega}_{\frac{n}{k}, \frac{n}{k+1}}, & 1 \leq k \leq n-1, \\
		\norm{\omega}_1, & k = n.
	\end{cases}
\]
Furthermore, we define the \emph{altered Sobolev spaces of conformal exponent} $\cesobt(\wedge^k M)$ by
\[
	\cesobt(\wedge^k M) = \begin{cases}
		\left\{ \omega \in L^{\infty, \flat}(\wedge^0 M) \;\big\vert\; d\omega \in L^n(\wedge^1 M) \right\}, 
			& k=0, \\
		\cesob(\wedge^k M), & 1 \leq k \leq n-2, \\
		\big\{ \omega \in L^{\frac{n}{n-1}}(\wedge^{n-1} M) \;\big\vert\;d\omega \in L^{1, \sharp}(\wedge^n M) \big\}, 
			& k = n-1, \\
		L^{1, \sharp}(\wedge^n M), & k = n, \\
		0, & k>n.
	\end{cases}
\]
Heuristically, the altered spaces have a flattened $L^\infty$-space for $0$-forms, and a sharpened $L^1$-space for $n$-forms.

Both the spaces $\cesob(\wedge^* M)$ and $\cesobt(\wedge^* M)$ form a chain complex with the weak exterior derivative $d$ as the boundary map. The reason for introducing the altered spaces is that, while the $k^\text{th}$ cohomology
\[
	H^k(\cesob(\wedge^{*} M)) = 
	\dfrac{
		\ker\left(d\colon \cesob(\wedge^{k} M) \to \cesob(\wedge^{k+1} M)\right)
	}{
		\im\left(d\colon \cesob(\wedge^{k-1} M) \to \cesob(\wedge^{k} M)\right)
	}
\]
of the complex $\cesob(\wedge^* M)$ is isomorphic with the $k^\text{th}$ real singular cohomology $H^k(M; \R)$ of $M$ for $k \in \{2, \ldots, n-1\}$, this isomorphism does not hold for $k =1$ due to failure of the Sobolev-Poincar\'e inequality; see \cite[Theorem 7.5 and Corollary 7.11]{GoldsteinTroyanov2010paper}. For the altered complex, however, this isomorphism holds for all $k \in \{0, \ldots, n\}$.

\begin{lemma}[{\cite[Theorem 4.1]{KangasniemiPankka2017paper}}]\label{lemma:sobolev_de_rham_isomorphism}
	Let $M$ be a closed $n$-manifold and $k \geq 0$. Let $\cehom{k}(M)$ denote the $k^\text{th}$ cohomology of the chain complex $\cesobt(\wedge^* M)$, given by
	\[
		\cehom{k}(M) = \dfrac{
			\ker\left(d\colon \cesobt(\wedge^{k} M) \to \cesobt(\wedge^{k+1} M)\right)
		}{
			\im\left(d\colon \cesobt(\wedge^{k-1} M) \to \cesobt(\wedge^{k} M)\right)
		}.
	\]
	Then there is a natural isomorphism $\cehom{k}(M) \cong H^k(M; \R)$.
\end{lemma}
The proof is a standard sheaf theoretic proof of the de Rham theorem; see \cite[Chapter 4]{KangasniemiPankka2017paper} for details.

We define the \emph{complex altered Sobolev spaces of conformal exponent}, denoted $\cesobt(\wedge^k M ; \C)$,  by
\begin{align*}
	&\cesobt(\wedge^k M; \C) \\
	&\quad= \begin{cases}
		\left\{ \omega \in L^{\infty, \flat}(\wedge^0 M; \C) \;\big\vert\; 
			d\omega \in L^n(\wedge^1 M; \C) \right\}, 
			& k=0, \\
		W^{d, \frac{n}{k}, \frac{n}{k+1}}(\wedge^k M; \C), & 1 \leq k \leq n-2, \\
		\big\{ \omega \in L^{\frac{n}{n-1}}(\wedge^{n-1} M; \C) \;\big\vert\;
			d\omega \in L^{1, \sharp}(\wedge^n M; \C) \big\}, 
			& k = n-1, \\
		L^{1, \sharp}(\wedge^n M; \C), & k = n, \\
		0, & k>n.
	\end{cases}
\end{align*} 
By the previous remarks on complex Lebesgue and Sobolev spaces, we have for all $k \geq 0$ the identity
\begin{align*}
	\cesobt(\wedge^k M ; \C) &= \cesobt(\wedge^k M) \otimes \C.
\end{align*}
Furthermore, a similar identity holds for the images and kernels of the boundary operator $d\colon \cesobt(\wedge^{k} M; \C) \to \cesobt(\wedge^{k+1} M; \C)$. Consequently, 
\begin{equation}\label{eq:complex_coefficients_in_conf_cohomology}
	\cehom{k}(M; \C) = \cehom{k}(M) \otimes \C,
\end{equation}
where $\cehom{k}(M; \C)$ is defined as the $k^\text{th}$ cohomology of the chain complex $\cesobt(\wedge^* M ; \C)$.

We require the following facts about $\cehom{*}$. 

\begin{lemma}\label{lemma:cohomology_class_closedness}
	Let $M$ be a closed manifold, and let $k \in \{1, \ldots, n-1\}$. Then $d\cesobt(\wedge^{k-1} M; \C)$ is a closed subspace of $L^p(\wedge^k M; \C)$.
\end{lemma}
\begin{proof}
	By \cite[Lemma 3.3]{KangasniemiPankka2017paper}, $d\cesobt(\wedge^{k-1} M)$ is a closed subspace of $L^p(\wedge^k M)$. Therefore, the lemma follows directly by \eqref{eq:complex_lp_estimate_finite_p}.
\end{proof}
\begin{lemma}\label{lemma:cohomology_class_smooth_elements}
	Let $M$ be a closed manifold, and let $k \in \{1, \ldots, n-1\}$. Every $c \in \cehom{k}(M; \C)$ contains a smooth element $\gamma \in C^\infty(\wedge^k M; \C)$.
\end{lemma}
\begin{proof}
	Let $c \in \cehom{k}(M; \C)$, and let $\omega \in c$. Then $\omega \in L^{n/k}(\wedge^k M)$. Consider the $L^{n/k}$-Hodge decomposition $\omega = d\alpha + d^*\beta + \gamma$ of $\omega$. Since $d\omega = 0$, $dd^*\beta = 0$, and by uniqueness of the Hodge decomposition on $M$, $d^*\beta = 0$. By the Sobolev-Poincar\'e inequality \eqref{eq:sobolev_poincare_ineq}, we may assume that $\alpha \in \cesobt(\wedge^{k-1} M)$. Hence, $\gamma = \omega - d\alpha \in c$. Since $\gamma$ is smooth, we obtain the claim.
\end{proof}

\subsection{Quasiregular mappings and push-forward}

Recall that a continuous map $f \colon M \to N$ between two closed $n$-manifolds is \emph{$K$-quasiregular} if $f$ is contained in the Sobolev space $W^{1, n}(M, N)$ and satisfies
\[
	\norm{Df(x)} \leq K J_f(x)
\]
for almost every $x \in M$ in the Lebesgue sense. Here, the Sobolev space $W^{1, n}(M, N)$ is defined using a smooth isometric Nash embedding $\iota \colon N \to \R^l$ for some $l > 0$, where $f \colon M \to N$ is an element of $W^{1, n}(M, N)$ if $\iota \circ f \in W^{1, n}(M; \R^l)$. The map $Df$ is then obtained as the unique bundle map $TM \to TN$ satisfying $D(\iota \circ f) = D\iota \circ Df$. See e.g. Haj\l asz--Iwaniec--Mal\'y--Onninen \cite{HajlaszIwaniecMalyOnninen2008paper} for more details on Sobolev spaces of mappings between closed manifolds.

Furthermore, recall that the \emph{degree} of a continuous map $f \colon M \to N$ between closed manifolds is the unique integer $\deg f$ satisfying $f^*c_N = (\deg f)c_M$, where $c_M$ and $c_N$ are the positively oriented generators of the compactly supported Alexander--Spanier cohomology groups $H^k_c(M; \Z)$ and $H^k_c(N; \Z)$, respectively. If $f \colon M \to N$ is a non-constant quasiregular map between closed manifolds, the degree of $f$ is always positive.

The main reason to introduce the conformal cohomology $\cehom{*}$ stems from the following lemma. For the proof, we refer to \cite[Lemma 3.4]{KangasniemiPankka2017paper}. 
\begin{lemma}[{\cite[Lemma 3.4]{KangasniemiPankka2017paper}}]\label{lemma:sobolev_de_rham_qr_pullback}
	Let $f \colon M \to N$ be a non-constant quasiregular map between closed $n$-manifolds. Then $f$ induces a pull-back chain map $f^* \colon \cesobt(\wedge^* N) \to \cesobt(\wedge^* M)$ satisfying $f^* \circ d = d \circ f^*$. Consequently, $f$ induces a natural pull-back map $f^* \colon \cehom{*}(N) \to \cehom{*}(M)$ in conformal cohomology.
\end{lemma}

Note that the conclusion of Lemma \ref{lemma:sobolev_de_rham_qr_pullback} generalizes to complex forms and complex cohomology, where the pull-back map $f^* \colon \cesobt(\wedge^* N; \C) \to \cesobt(\wedge^* M; \C)$ is defined componentwise by
\[
	f^*(\alpha + i\beta) = f^*\alpha + if^*\beta
\]
for $\alpha + i\beta \in \cesobt(\wedge^* N; \C)$.

Next, we recall the quasiregular push-forward map of \cite[Section 5]{KangasniemiPankka2017paper}. Given a quasiregular map $f \colon M \to N$ between closed $n$-manifolds, by applying the Vitali covering theorem we obtain an open set $V_f \subset N$ and disjoint open sets $U_{f, 1}, \ldots, U_{f, \deg f} \subset M$ which cover $N$ and $M$ up to a Lebesgue null-set respectively, and on which $f$ has well-defined quasiconformal inverses $f_i^{-1} \colon V_f \to U_{f, i}$. Note that the sets $V_f$ and $U_{f, i}$ are not domains, but instead countable unions of pairwise disjoint open balls. The \emph{quasiregular push-forward} $\push{f} \colon \Gamma(\wedge^k M) \to \Gamma(\wedge^k N)$ is now defined for $\omega \in \Gamma(\wedge^k M)$ by
\[
	(\push{f} \omega)\vert V_f = \sum_{i=1}^{\deg f} (f_i^{-1})^* \omega;
\]
note that it is sufficient to define $\push{f}\omega$ in $V_f$, since $V_f$ is of full measure.

The main properties of the quasiregular push-forward are the following; see \cite[Theorem 5.3, Corollary 5.4 and Lemma 5.5]{KangasniemiPankka2017paper} for proofs.
\begin{lemma}\label{lemma:qr_pushforward_properties}
	Let $f \colon M \to N$ be a quasiregular map between closed $n$-manifolds, and let $k, l \in \{0, \ldots n\}$ be indices satisfying $k+l \leq n$. Then
	\begin{enumerate}
		\item for all $\alpha \in \Gamma(\wedge^k M)$ and $\beta \in \Gamma(\wedge^{l } M)$,
			\[
				\push{f}(\alpha \wedge f^*\beta) = (\push{f} \alpha) \wedge \beta;
			\]
		\item for all $\omega \in \Gamma(\wedge^k M)$,
			\[
				\push{f} f^* \omega = (\deg f)\omega;
			\]
		\item for all $\omega \in L^1(\wedge^n M)$,
			\[
				\int_N \push{f} \omega = \int_M \omega;
			\]
		\item for all $\omega \in \cesobt(\wedge^k M)$, $\push{f} \omega \in \cesobt(\wedge^k M)$ and $d\push{f} \omega = \push{f} d\omega$;
		\item the push-forward induces a natural map $\push{f} \colon \cehom{*}(M) \to \cehom{*}(N)$ in the conformal cohomology.
	\end{enumerate}
\end{lemma}

As with the pull-back map $f^* \colon \Gamma(\wedge^k N; \C) \to \Gamma(\wedge^k M; \C)$, the quasiregular push-forward admits a complex generalization $\push{f} \colon \Gamma(\wedge^k M; \C) \to \Gamma(\wedge^k N; \C)$, which is defined componentwise by
\[
	\push{f}(\alpha + i\beta) = \push{f}\alpha + i\push{f}\beta
\]
for $\alpha, \beta \in \Gamma(\wedge^* M)$. Furthermore, all the properties discussed in Lemma \ref{lemma:qr_pushforward_properties} clearly generalize to the complex case.

%%%%%%%%%%%%%%%%%%%%%%%%%%%%%%%%%%%%%%%%%%%%%%%%%%%%%%%%%%%%%%%%%%%%%%%%%%%%%%%%%%%%%%%%%%%%%%%%%%
%%%%%%%%%%%%%%%%%%%%%%%%%%%%%%%%%%%%%%%%%%%%%%%%%%%%%%%%%%%%%%%%%%%%%%%%%%%%%%%%%%%%%%%%%%%%%%%%%%

\section{Representation of cohomology classes of $\cehom{*}(M; \C)$}\label{sect:cohomology_representation}

In this section, we first recall the invariant conformal structure of Iwaniec and Martin \cite{IwaniecMartin1996paper} for a uniformly quasiregular map $f \colon M \to M$. Then we use this structure to obtain a representation of cohomology classes by measurable forms which is closed under pull-back by a given uniformly quasiregular self-map.

\subsection{Invariant conformal structure}

Given an $n$-dimensional inner product space $V$, let $S(V)$ denote the space of linear self-maps $V \to V$ which are positive-definite, symmetric, and have determinant 1. Note that the definition of $S(V)$ depends only on the inner product of $V$, and doesn't require fixing a basis of $V$. The space $S(V)$ admits a metric $\rho_V$ satisfying 
\[
	\rho_V(A, \id_{V}) = \sqrt{\dfrac{(\log \lambda_1)^2 + \ldots + (\log \lambda_n)^2}{n}} 
	\quad \text{for } A \in S(V), 
\]
where $\lambda_i$ are the eigenvalues of $A$; see Iwaniec--Martin \cite[Section 20.1]{IwaniecMartin2001book} for details.

The spaces $S(T_x M)$ form a fiber bundle $S(TM)$ over $M$, which is topologized by using smooth local orthonormal frames on $M$ to locally identify $S(TM)$ with $M \times S(\R^n)$. A \emph{conformal structure} on $M$ is a measurable section $G \colon M \to S(TM)$ which is essentially bounded, that is,
\[
	\esssup_{x \in M} \rho_{T_x M}(G(x), \id_{T_x M}) < \infty.
\]
Given a continuous $W^{1, n}$-mapping $f \colon M \to M$, we say that a conformal structure $G$ on $M$ is \emph{$f$-invariant} if it satisfies
\begin{equation}\label{eq:f-invariance_definition}
	(Df(x))^T G(f(x)) Df(x) = J_f^\frac{2}{n}(x) \, G(x)
\end{equation}
for almost every $x \in M$.

The existence of an $f$-invariant conformal structure for a non-constant uniformly quasiregular $f$ is due to Iwaniec and Martin, based on a similar construction of Tukia \cite[Theorem F]{Tukia1993paper} in the quasiconformal case.

\begin{thm}[{\cite[Theorem 5.1]{IwaniecMartin1996paper}}]\label{thm:iwaniec_martin_invariant_conformal_structure}
	Let $M$ be a closed manifold and let $f \colon M \to M$ be a non-constant uniformly quasiregular map. Then there exists an $f$-invariant conformal structure $G_f$ on $M$.
\end{thm}

For a proof of Theorem \ref{thm:iwaniec_martin_invariant_conformal_structure}, see \cite[Section 5]{IwaniecMartin1996paper}, or alternatively \cite[Section 21.5]{IwaniecMartin2001book}. Note that the proof is written for domains $\Omega$ in $\S^n$ using the canonical identification $T_x \S^n \cong \R^n$ for $x \in \S^n$, whereby conformal structures may be defined as just measurable essentially bounded maps $\Omega \to S(\R^n)$. However, the proof generalizes in a straightforward manner to the above definition of conformal structures on closed manifolds. See also \cite[Section 4]{OkuyamaPankka2014paper} for further discussion.

Let $f \colon M \to M$ be a uniformly quasiregular and non-constant map. The invariant conformal structure $G_f$ defines a measurable Riemannian metric $\ip{\cdot}{\cdot}_f$ on $M$ by
\begin{equation}\label{eq:invariant_riemannian_metric}
	\ip{a}{b}_f = \ip{G_f(x)a}{b}
\end{equation}
for almost every $x \in M$ and all $a, b \in T_x M$. Using the bundle map $\sigma_f \colon TM \to T^*M$, defined by $\sigma_f(v)(w) = \ip{v}{w}_f$ for $v, w \in T_x M$, the Riemannian metric $\ip{\cdot}{\cdot}_f$ extends to the coexterior bundles $\wedge^k T^*M$ by equation \eqref{eq:grassmann}. By a straightforward calculation using \eqref{eq:f-invariance_definition},  \eqref{eq:invariant_riemannian_metric}, and \eqref{eq:grassmann}, the formula
\begin{equation}\label{eq:inv_product_pullback_formula}
	\ip{(f^*\omega)_x}{(f^*\tau)_x}_f = \ip{\omega_{f(x)}}{\tau_{f(x)}}_f J_f^\frac{2k}{n}(x),
\end{equation}
holds for almost every $x \in M$ and all $\omega, \tau \in \Gamma(\wedge^k M)$, where $0 < k \leq n$. By \eqref{eq:complex_inner_product}, the formula \eqref{eq:inv_product_pullback_formula} also generalizes to $\omega$ and $\tau$ with complex coefficients.

As in the case of the original Riemannian metric on $M$, the measurable Riemannian metric $\ip{\cdot}{\cdot}_f$ induces measurable pointwise norms $\abs{\cdot}_f$ on the coexterior bundles $\wedge^k T^*M$, and $L^p$-norms $\norm{\cdot}_{f, p}$ on the spaces of measurable $k$-forms $\Gamma(\wedge^k M)$ and $\Gamma(\wedge^k M; \C)$. Due to the essential boundedness of $G_f$, we have the estimate
\begin{equation}\label{eq:norm_comparability}
	C^{-1} \abs{\omega_x} \leq \abs{\omega_x}_f \leq C\abs{\omega_x},
\end{equation}
for $\omega \in \Gamma(M; \C)$ and almost every $x \in M$; here, the constant $C$ depends only on the essential bound of $G_f$, which in turn depends only on the dimension $n$ of $M$ and the distortion constant $K$ of $f$. Hence, the norm $\norm{\cdot}_{f, p}$ is equivalent with the standard $L^p$-norm on the spaces $L^p(\wedge^k M)$ and $L^p(\wedge^k M; \C)$, $1 \leq p \leq \infty$.

Let $0 < k \leq n$, and let $\omega \in L^{n/k}(\wedge^k M; \C )$. As a consequence of \eqref{eq:inv_product_pullback_formula},
\begin{equation}\label{eq:inv_pointwise_pullback_formula}
	\abs{f^* \omega}_f = \left(\abs{\omega}_f \circ f\right) J_f^{\frac{k}{n}}
\end{equation}
almost everywhere on $M$. By the quasiregular change of variables, we obtain
\begin{equation}\label{eq:inv_lp_pullback_formula}
\norm{f^* \omega}_{f, \frac{n}{k}} = (\deg f)^\frac{k}{n} \norm{\omega}_{f, \frac{n}{k}}.
\end{equation}
In this way, for $0 < k \leq n$, the $f$-invariant conformal structure yields an equivalent norm on $L^{n/k}(\wedge^k M)$ and $L^{n/k}(\wedge^k M; \C)$ under which the pull-back $f^*$ is uniformly expanding.

Finally, \eqref{eq:inv_lp_pullback_formula} has a counterpart for the quasiregular push-forward.

\begin{lemma}\label{lem:inv_lp_pushforward}
	Let $f \colon M \to M$ be a non-constant uniformly quasiregular map, $0 < k \leq n$, and let $\omega \in L^{n/k}(\wedge^k M)$. Then
	\begin{equation}\label{eq:inv_lp_pushforward_formula}
		\norm{\push{f} \omega}_{f, \frac{n}{k}} \leq (\deg f)^\frac{n-k}{n} \norm{\omega}_{f, \frac{n}{k}}.
	\end{equation}
\end{lemma}
\begin{rem}
	Note that there is no lower bound of the form $\norm{\push{f} \omega}_{f, n/k} \geq C^{-1}(\deg f)^{(n-k)/n} \norm{\omega}_{f, n/k}$ for some $C \geq 1$, since the quasiregular push-forward is not in general injective.
\end{rem}
\begin{proof}[Proof of Lemma \ref{lem:inv_lp_pushforward}]
	Let $V_f \subset M$ and $U_{f, i} \subset M$ be as in the definition of the quasiregular push-forward, alongside the maps $f^{-1}_i \colon V_f \to U_{f, i}$. Since the restriction $f\vert U_{f, i} \colon U_{f, i} \to V_f$ is quasiconformal, we have by \eqref{eq:inv_pointwise_pullback_formula} and the quasiconformal change of variables that 
	\begin{align*}
		\int_{V_f} \abs{(f_i^{-1})^* \omega}_f^\frac{n}{k} \vol_M
			&= \int_{U_{f, i}} \left(\abs{(f_i^{-1})^* \omega}_f^\frac{n}{k} \circ f\right) J_f \vol_M\\
			&= \int_{U_{f, i}} \abs{f^*(f_i^{-1})^* \omega}_f^\frac{n}{k} \vol_M\\
			&= \int_{U_{f, i}} \abs{\omega}_f^\frac{n}{k} \vol_M.
	\end{align*}
	Hence, a calculation similar to the proof of \cite[Lemma 5.8]{KangasniemiPankka2017paper} yields
	\begin{align*}
		\norm{\push{f} \omega}_{f, \frac{n}{k}}
		&= \left( \int_{V_f} \abs{\sum_{i=1}^{\deg f} \left(f_i^{-1}\right)^*\omega}^\frac{n}{k}_f
			 \vol_N \right)^\frac{k}{n}\\
		&\leq \left((\deg f)^{\frac{n}{k}-1} \sum_{i=1}^{\deg f} \int_{V_f} 
		\abs{\left(f_i^{-1}\right)^*\omega}^\frac{n}{k}_f \vol_N \right)^\frac{k}{n}\\
		&= (\deg f)^\frac{n-k}{n} \left( \sum_{i=1}^{\deg f} \int_{U_{f, i}} 
		\abs{\omega}^\frac{n}{k}_f \vol_N \right)^\frac{k}{n}\\
		&= (\deg f)^\frac{n-k}{n} \norm{\omega}_{f, \frac{n}{k}}.
	\end{align*}
	This completes the proof. 
\end{proof}

\subsection{The $f$-harmonic representation of $\cehom{*}(M; \C)$}\label{subsect:cohomology_representation}

In this section, we consider elements of cohomology classes $c \in \cehom{k}(M; \C)$ which minimize the norm $\norm{\cdot}_{f, n/k}$. We refer to Iwaniec--Scott--Stroffolini \cite[Section 7.1]{IwaniecScottStroffolini1999paper} or Bonk--Heinonen \cite[Section 3]{Bonk-Heinonen_Acta} for further discussion regarding cohomological norm-minimizers.

Let $f \colon M \to M$ be a non-constant uniformly quasiregular self-map on a closed manifold $M$, and let $0 < k < n$. By Lemma \ref{lemma:cohomology_class_closedness}, every cohomology class $c \in \cehom{k}(M; \C)$ is a closed affine subspace of  $L^{n/k}(\wedge^k M; \C)$. Furthermore, the Banach space $(L^{n/k}(\wedge^k M; \C), \norm{\cdot}_{f, n/k})$ is uniformly convex by the classical proof involving Hanner's inequalities; see e.g.\ \cite{Hanner1956paper}. Consequently, for every $c \in \cehom{k}(M; \C)$ there exists a unique form $\omega_{c} \in c$ satisfying 
\[
	\norm{\omega_{c}}_{f, \frac{n}{k}} = \inf_{\omega \in c} \norm{\omega}_{f, \frac{n}{k}}.
\]

Following the example of representing cohomology classes in the de Rham cohomology by harmonic forms, we define the space of \emph{$f$-harmonic complex $k$-forms} $\fharm{k}(M; \C)$ by
\[
	\fharm{k}(M; \C) = \left\{ \omega_{c} \in c \colon c \in \cehom{k}(M; \C), \, \norm{\omega_{c}}_{f, \frac{n}{k}} = \inf_{\omega \in c} \norm{\omega}_{f, \frac{n}{k}} \right\}.
\] 
The map $c \mapsto \omega_c$ defines a bijection $\cehom{k}(M; \C) \to \fharm{k}(M; \C)$. The forms $\omega_c \in \fharm{k}(M; \C)$ satisfy 
\begin{equation}\label{eq:minimizers_multiplication}
	\lambda\omega_{c} = \omega_{\lambda c} \in  \fharm{k}(M; \C)
\end{equation}
for every $\lambda \in \C$. Note, however, that in general $\omega_{c+c'}$ does not equal $\omega_{c} + \omega_{c'}$ for $c, c' \in \cehom{k}(M; \C)$, unless $k = n/2 \in \Z$.

\begin{prop}
	Let $f \colon M \to M$ be a non-constant uniformly quasiregular map on a closed manifold $M$, let $0 < k < n$, and let $\omega_c \in \fharm{k}(M; \C)$. Then
	\begin{equation}\label{eq:minimizers_pullback}
		f^*\omega_{c} = \omega_{f^*c} \in \fharm{k}(M; \C).
	\end{equation}
\end{prop}
\begin{proof}
	Note that $(\deg f)^{-1} \push{f} (\omega_{f^*c}) \in (\deg f)^{-1} \push{f} f^*c = c$. Hence, by \eqref{eq:inv_lp_pullback_formula} and \eqref{eq:inv_lp_pushforward_formula}, 
	\begin{align*}
		\norm{f^*\omega_{c}}_{f, \frac{n}{k}}
		&= (\deg f)^\frac{k}{n}\norm{\omega_{c}}_{f, \frac{n}{k}}
		\leq (\deg f)^\frac{k}{n} \norm{(\deg f)^{-1} \push{f} (\omega_{f^*c})}_{f, \frac{n}{k}}\\
		&= (\deg f)^{-\frac{n-k}{n}} \norm{\push{f} (\omega_{f^*c})}_{f, \frac{n}{k}}
		\leq \norm{\omega_{f^*c}}_{f, \frac{n}{k}}.
	\end{align*}
	Since $f^*\omega_c \in f^*c$, the uniqueness of the norm-minimizers implies that $f^*\omega_{c} = \omega_{f^*c}$.
\end{proof}

Note that \eqref{eq:minimizers_multiplication} and \eqref{eq:minimizers_pullback} immediately yield the following corollary, which is crucial for the proof of the main results.

\begin{cor}\label{cor:eigenvalues_translate_to_forms}
	Let $f \colon M \to M$ be a non-constant uniformly quasiregular map on a closed manifold $M$, let $0 < k < n$, and let $\omega_c \in \fharm{k}(M; \C)$. Assume that $f^*c = \lambda c$ for some $\lambda \in \C$. Then
	\[
		f^*\omega_{c} = \lambda\omega_{c}.
	\]
\end{cor}

\subsection{Higher integrability of $f$-harmonic forms}
In what follows, we use the following higher integrability result for the norm minimizing forms $\omega_c \in \fharm{k}(M; \C)$.

\begin{prop}\label{prop:higher_integrability}
	Let $f \colon M \to M$ be a non-constant uniformly quasiregular map, and let $0 < k < n$. Then
	\[
		\fharm{k}(M; \C) \subset L^{\frac{n}{k}, \sharp}(\wedge^k M; \C).
	\]
\end{prop}

In the case of forms with real coefficients, Proposition \ref{prop:higher_integrability} follows directly from more general results of Iwaniec, Scott, and Stroffolini; see \cite[Theorems 7.2, 8.4, 9.1, and Remark 9.6]{IwaniecScottStroffolini1999paper}. For $k=n/2 \in \Z$, the complex case follows directly from the real case, since the norm-minimization may be done for the real and imaginary parts seperately due to linearity. For $k \neq n/2$, we prove Proposition \ref{prop:higher_integrability} by following the original proof of Iwaniec, Scott, and Stroffolini.

We use the notation 
\begin{align*}
&\ip{\cdot}{\cdot}_\R = \Re \ip{\cdot}{\cdot} &&\text{and} && \ip{\cdot}{\cdot}_{f, \R} = \Re \ip{\cdot}{\cdot}_f
\end{align*}
for the induced real inner product. Note that these induce the same norms as their complex counterparts, since $\Im \ip{\omega_x}{\omega_x} = 0$ for every $\omega \in \Gamma(M; \C)$ and $x \in M$. Furthermore, we use the notation $G_f^{k,*}$ for the mapping $(\wedge^k T^*M) \otimes \C \to (\wedge^k T^*M) \otimes \C$ defined by
\[
	G_f^{k,*} = \wedge^k \left( \sigma \circ G_f^{-1} \circ \sigma^{-1} \right),
\]
where $\sigma \colon TM \to T^*M$ is the natural bundle map defined in Section \ref{subsect:diff_forms}. The reason to define $G_f^{k,*}$ is that, by \eqref{eq:grassmann} and \eqref{eq:invariant_riemannian_metric}, the inner product $\ip{\cdot}{\cdot}_f$ takes then the form
\[
	\ip{\alpha}{\beta}_f = \ip{G_f^{k, *} \alpha}{\beta}
\]
for $\alpha, \beta \in (\wedge^k T^*_xM) \otimes \C$ and almost every $x \in M$. We now define the operator $\cG_k \colon \Gamma(\wedge^k M; \C) \to \Gamma(\wedge^k M; \C)$ by
\[
	\cG_k(\omega) = \abs{\omega}_f^{\frac{n}{k} - 2} G_f^{k,*} \omega
\]
for $\omega \in \Gamma(\wedge^k M; \C)$.

Proposition \ref{prop:higher_integrability} reduces to the following lemma; for the corresponding stronger version for forms with real coefficients, see \cite[Theorems 8.4 and 9.1]{IwaniecScottStroffolini1999paper}. 

\begin{lemma}\label{lemma:higher_integrable_hodge_conjugates}
	Let $0 < k < n$, and let 
	\begin{align*}
		&\phi_0 \in L^{\frac{n}{k}, \sharp}(\wedge^k M; \C), &&\text{and} && \psi_0 \in L^{\frac{n}{n-k}, \sharp}(\wedge^k M; \C).
	\end{align*}
	Then there exist
	\begin{align*}
		&\phi \in L^{\frac{n}{k}, \sharp}(\wedge^k M; \C) && \text{and} &&  \psi \in L^{\frac{n}{n-k}, \sharp}(\wedge^k M; \C)
	\end{align*}
	satisfying 
	\begin{align}\label{eq:hodge_problem}
		\phi \in \Im d, 
		&&  d^*\psi = 0, && \text{and} 
		&& \cG_k(\phi+\phi_0) = \psi+\psi_0,
	\end{align}
	where $d$ and $d^*$ are the weak exterior and coexterior derivatives, respectively.
\end{lemma}
\begin{proof}
	[Proof of Proposition \ref{prop:higher_integrability} assuming Lemma \ref{lemma:higher_integrable_hodge_conjugates}]
	Let $c \in \cehom{k}(M; \C)$. By Lem\-ma \ref{lemma:cohomology_class_smooth_elements}, we may fix a smooth $k$-form $\omega_0 \in c$. Let $\phi$ and $\psi$ be as in Lemma \ref{lemma:higher_integrable_hodge_conjugates} with $\phi_0 = \omega_0$ and $\psi_0 = 0$. Since $\phi \in \Im d$, we have by the Sobolev--Poincare inequality \eqref{eq:sobolev_poincare_ineq} that $\phi \in d\cesobt(\wedge^{k-1} M; \C)$. Hence $\phi + \phi_0 \in c$.
	
	Let $\tau \in \cesobt(\wedge^{k-1} M; \C)$. Since $d^*\psi = 0$, we have that
	\begin{equation}\label{eq:hodge_system_vanishing_inner_product}
		\int_M \ip{\psi}{d\tau} \vol_M 
		= \int_M \ip{d^*\psi}{\tau} \vol_M = 0.
	\end{equation}
	Let $\omega \in c$. Then there exists $\tau \in \cesobt(\wedge^{k-1} M; \C)$ for which $\omega = \phi + \phi_0 +d\tau$. By \eqref{eq:hodge_system_vanishing_inner_product}, we have
	\begin{align*}
		\norm{\phi + \phi_0}_{f, \frac{n}{k}}
		&= \left( \int_M \ip{\abs{\phi+\phi_0}_f^{\frac{n}{k} - 2} (\phi + \phi_0)\,}{\phi + \phi_0}_f 
			\vol_M \right)^\frac{k}{n}\\
		&= \left( \int_M \ip{\abs{\phi+\phi_0}_f^{\frac{n}{k} - 2} (\phi + \phi_0)\,}{\omega}_f 
			\vol_M \right)^\frac{k}{n}\\
		&\leq \left( \int_M \abs{\phi+\phi_0}_f^{\frac{n-k}{k}}\abs{\omega}_f \vol_M \right)^\frac{k}{n}\\
		&\leq \left(\norm{\phi + \phi_0}_{f, \frac{n}{k}}\right)^\frac{n-k}{n}
			\left(\norm{\omega}_{f, \frac{n}{k}}\right)^\frac{k}{n}.
	\end{align*}
	Hence $\norm{\phi + \phi_0}_{f, n/k } \leq \norm{\omega}_{f, n/k }$, and $\phi + \phi_0 = \omega_c$ by the uniqueness of the norm-minimizer. The claim now follows, since $\phi + \phi_0 \in L^{n/k, \sharp}(\wedge^k M; \C)$.
\end{proof}

\begin{proof}[Proof of Lemma \ref{lemma:higher_integrable_hodge_conjugates}]
	The proof of Lemma \ref{lemma:higher_integrable_hodge_conjugates} is for the most part identical with the corresponding proof in Iwaniec--Scott--Stroffolini \cite{IwaniecScottStroffolini1999paper} for the case of forms with real coefficients. Hence, we only sketch the proof, with special attention on the minor differences caused by the complex coefficients. 
	
	We claim that $\cG_k$ satisfies the conditions specified in \cite[(8.15)--(8.17)]{IwaniecScottStroffolini1999paper}. More precisely, given $0 < k < n$ and $p = n/k$, there exists a constant $C_\cG = C_\cG(n, k) \geq 1$ with the property that whenever $\xi, \zeta \in L^{n/k}(\wedge^k M)$, we have
	\begin{equation}\label{eq:hodge_field_1}
		\abs{\cG_k(\xi) - \cG_k(\zeta)} 
		\leq C_\cG(\abs{\xi} + \abs{\zeta})^{p - 2}\abs{\xi-\zeta},
	\end{equation}
	\begin{equation}\label{eq:hodge_field_2}
		\ip{\cG_k(\xi) - \cG_k(\zeta)}{\xi-\zeta}_\R 
		\geq C_\cG^{-1}(\abs{\xi} + \abs{\zeta})^{p - 2}\abs{\xi-\zeta}^2,
	\end{equation}
	and
	\begin{equation}\label{eq:hodge_field_3}
		\cG_k(t\xi) 
		= t\abs{t}^{p-2}\cG_k(\xi)
	\end{equation}
	almost everywhere on $M$. Note that condition \eqref{eq:hodge_field_2} is for the induced real inner product, since $\ip{\cG_k(\xi) - \cG_k(\zeta)}{\xi-\zeta}$ may have an imaginary part. 
	
	It is obvious that $\cG_k$ satisfies condition \eqref{eq:hodge_field_3}. By \eqref{eq:norm_comparability}, the conditions \eqref{eq:hodge_field_1} and \eqref{eq:hodge_field_2} for $\cG_k$ reduce to the corresponding conditions for the classical operators $H_p \colon \omega \mapsto \abs{\omega}^{p-2}\omega$, where $1 < p < \infty$. For the proof of condition \eqref{eq:hodge_field_2} for $H_p$, see e.g.\ \cite[Lemma 4 and Corollary 5]{HolopainenPigolaVeronelli2011paper}. Condition \eqref{eq:hodge_field_1} for $H_p$ reduces to condition \eqref{eq:hodge_field_2} for $H_q$, where $1/p + 1/q = 1$; see e.g.\ \cite[Section 12, footnote 16]{Lindqvist2017notes} for the general idea.
	
	Next, we show the existence of solutions $\phi \in L^{n/k}(\wedge^k M; \C)$ and $\psi \in L^{n/(n-k)}(\wedge^k M; \C)$ which satisfy \eqref{eq:hodge_problem}; see \cite[Theorem 8.4]{IwaniecScottStroffolini1999paper}. For this, let $E \colon L^{n/(n-k)}(\wedge^k M; \C) \to dW^{d,n/(n-k)}(\wedge^k M; \C)$ be the operator mapping a form $\xi$ to the exact part $d\alpha$ of its Hodge decomposition $\xi = d\alpha + d^*\beta + \gamma$. Consider the operator $\cE \colon L^{n/k}(\wedge^k M; \C) \to L^{n/(n-k)}(\wedge^k M; \C)$ given by $\xi \mapsto E\cG_k(\xi+\phi_0)$ for $\xi \in L^{n/k}(\wedge^k M; \C)$. It suffices now to show that there exist $\phi \in L^{n/k}(\wedge^k M; \C)$ and $\psi \in L^{n/(n-k)}(\wedge^k M; \C)$ satisfying $\cE (\phi) = E(\psi_0)$ and $\psi = \cG_k(\phi+\phi_0) - \psi_0$. 
	
	It is enough to show that $\cE$ is surjective. This is proven by using the Browder--Minty theorem for complex Banach spaces; see \cite[Theorem 2]{Browder1963paper}. The Browder--Minty theorem requires that $L^{n/k}(\wedge^k M; \C)$ is reflexive, separable, and that its continuous dual is $L^{n/(n-k)}(\wedge^k M; \C)$. By \eqref{eq:complex_lp_estimate_finite_p}, these conditions follow from the corresponding properties in the real case. We also require that the operator $\cE$ is continuous, strictly monotone, and coercive. The verification is based on conditions \eqref{eq:hodge_field_1}--\eqref{eq:hodge_field_3}, and is essentially identical to the one for the real counterpart of $\cE$ by Iwaniec, Scott and Stroffolini \cite[Theorem 8.4]{IwaniecScottStroffolini1999paper}. Note that in our case the verification uses the continuity of the complex exact projection $E$, but this again reduces to the real case \cite[Proposition 5.5]{IwaniecScottStroffolini1999paper} by \eqref{eq:complex_lp_estimate_finite_p}. Hence, the map $\cE$ is surjective, and there exist solutions $\phi \in L^{n/k}(\wedge^k M; \C), \psi \in L^{n/(n-k)}(\wedge^k M; \C)$ for \eqref{eq:hodge_problem}.
	
	Next, the estimate
	\begin{equation}\label{eq:hodge_field_norm_estimate}
		\int_U (\abs{\phi}^\frac{n}{k} + \abs{\psi}^{\frac{n}{n-k}}) \vol_M
		\leq C \int_U (\abs{\phi_0}^\frac{n}{k} + \abs{\psi_0}^\frac{n}{n-k}) \vol_M
	\end{equation}
	is estabilished for $\phi, \phi_0 \in L^{n/k}(\wedge^k U; \C)$ and $\psi, \psi_0 \in L^{n/(n-k)}(\wedge^k U; \C)$ satisfying \eqref{eq:hodge_problem} in a domain $U \subset M$; see \cite[Theorem 8.4]{IwaniecScottStroffolini1999paper}. The proof is by straightforward estimates, and the only notable difference between the real and complex cases is use of the inner product $\ip{\cdot}{\cdot}_\R$. By using the complex version of the Hodge decomposition \eqref{eq:hodge_decomposition}, the estimate \eqref{eq:hodge_field_norm_estimate} yields a Caccioppoli-type inequality for such $\phi, \psi, \phi_0$ and $\psi_0$; see \cite[Theorem 8.8]{IwaniecScottStroffolini1999paper}.
	
	What remains is to follow the proof of \cite[Theorem 9.1]{IwaniecScottStroffolini1999paper} up to the end of what is labelled as \emph{Step 1}. There, by using a suitable chart $\R^n \to U \subset M$ and a compactly supported function $\eta \in C^\infty_0(U)$, the problem is first reduced to a version where the forms $\phi, \psi, \phi_0$ and $\psi_0$ are compactly supported measurable forms on $\R^n$, which in our case have complex coefficients. Then, the Caccioppoli-type inequality is used to derive the reverse Hölder -type inequality
	\begin{equation}\label{eq:reverse_hölder}
		\left(\frac{1}{m_n(Q)}\int_Q F^r  \right)^\frac{1}{r}
		\leq \frac{A}{m_n(2Q)}\int_{2Q} F 
		+ \left(\frac{B}{m_n(2Q)}\int_{2Q} F_0^r  \right)^\frac{1}{r},
	\end{equation}
	where $Q$ is an arbitrary cube on $\R^n$, $F = (\smallabs{\phi}^{n/k} + \smallabs{\psi}^{n/(n-k)})^{1/r}$, $F_0 = (\smallabs{\phi_0}^{n/k} + \smallabs{\psi_0}^{n/(n-k)})^{1/r}$, and $r, A, B$ are constants independent of $Q$. The proof of \eqref{eq:reverse_hölder} uses the Sobolev-Poincar\'e inequality \eqref{eq:sobolev_poincare_ineq} on cubes; in our case, the version used is the one for forms with complex coefficients. Afterwards, the claim follows from Gehring's lemma, see e.g. \cite[Corollary 14.3.1]{IwaniecMartin2001book}.
\end{proof}

%%%%%%%%%%%%%%%%%%%%%%%%%%%%%%%%%%%%%%%%%%%%%%%%%%%%%%%%%%%%%%%%%%%%%%%%%%%%%%%%%%%%%%%%%%%%%%%%%%
%%%%%%%%%%%%%%%%%%%%%%%%%%%%%%%%%%%%%%%%%%%%%%%%%%%%%%%%%%%%%%%%%%%%%%%%%%%%%%%%%%%%%%%%%%%%%%%%%%

\section{Invariant measure}\label{sect:inv_equilibrium_measure}

In this section we recall the invariant measure of Okuyama and Pankka \cite{OkuyamaPankka2014paper} for uniformly quasiregular mappings, and give a streamlined proof for its existence \cite[Theorem 5.2]{OkuyamaPankka2014paper} using the conformal cohomology $\cehom{*}$. This result is slightly stronger than the original version.

\begin{thm}\label{thm:okuyama_pankka_complex_n-forms}
	Let $f \colon M \to M$ be a non-constant uniformly quasiregular self-map on a closed $n$-manifold $M$ satisfying $\deg f \geq 2$. Then, for $\omega \in L^{1, \sharp}(\wedge^n M; \C)$, there exists a complex-valued measure $\mu_\omega$ on $M$ for which
	\[
		\frac{(f^m)^* \omega}{(\deg f)^m} \xrightarrow[m \to \infty]{} \mu_\omega
	\]
	in the weak sense. Furthermore, given another $n$-form $\omega' \in L^{1, \sharp}(\wedge^n M; \C)$, the limit measures $\mu_\omega$ and $\mu_{\omega'}$ satisfy the uniqueness condition
	\[
		\left(\int_M \omega' \right) \mu_\omega = \left(\int_M \omega \right) \mu_{\omega'}.
	\]
\end{thm}

Recall that a sequence $(\mu_m)$ of complex-valued measures on a closed manifold $M$ \emph{converges weakly to a complex-valued measure $\mu$ on $M$} if
\[
	\int_M \eta d\mu_m \xrightarrow[m \to \infty]{} \int_M \eta d\mu
\]
for all smooth test functions $\eta \in C^\infty(M; \C)$, or equivalently for all $\eta \in C^\infty(M)$. Note that for a real $n$-form $\omega$ the limit measure $\mu_\omega$ in Theorem \ref{thm:okuyama_pankka_complex_n-forms} is clearly a real-valued signed measure.

The original result \cite[Theorem 5.2]{OkuyamaPankka2014paper} is for real probability measures, but generalizes easily to the complex case when applied to real and imaginary parts seperately. However, the minor improvement in Theorem \ref{thm:okuyama_pankka_complex_n-forms} is that $\omega$ is assumed to be in $L^{1, \sharp}(\wedge^n M)$, whereas in  \cite[Theorem 5.2]{OkuyamaPankka2014paper} $\omega$ is assumed to be in $L^p(\wedge^n M)$ for a given $p > 1$ depending only on $n$ and the distortion constant $K$ of $f$.

\begin{proof}[Proof of Theorem \ref{thm:okuyama_pankka_complex_n-forms}]
	Let $\omega \in L^{1, \sharp}(\wedge^n M; \C)$, and let $c = [\omega]$ be the cohomology class of $\omega$ in $\cehom{n}(M; \C)$. Since $f^*c = (\deg f)c$, there exists $\tau \in \cesobt(\wedge^{n-1}M; \C) \subset L^{n/(n-1)}(\wedge^kM; \C)$ for which 
	\[
		\frac{f^* \omega}{\deg f} - \omega = d\tau.
	\]
	
	Let $\eta \in C^\infty(M)$, and let $m \in \Z_+$. Then 
	\begin{align*}
		\abs{\int_M \eta\left(
			\frac{(f^{m+1})^* \omega}{(\deg f)^{m+1}} - \frac{(f^{m})^* \omega}{(\deg f)^{m}}\right)}
		&= \frac{1}{(\deg f)^m} \abs{\int_M \eta \, (f^m)^*d\tau}
	\end{align*}
	Since $f^* \colon \cesobt(\wedge^* M) \to \cesobt(\wedge^* M)$ is a chain map, we obtain
	\begin{align*}
		\abs{\int_M \eta \, (f^m)^*d\tau}
		=\abs{\int_M \eta \, d\left((f^m)^*\tau\right)}
		= \abs{\int_M d\eta \wedge (f^m)^*\tau}
	\end{align*}
	Now, by Hölder's inequality, \eqref{eq:norm_comparability} and \eqref{eq:inv_lp_pullback_formula}, we obtain the estimate
	\begin{align*}
		\abs{\int_M \eta\left(
			\frac{(f^{m+1})^* \omega}{(\deg f)^{m+1}} - \frac{(f^{m})^* \omega}{(\deg f)^{m}}\right)}
		&= \frac{1}{(\deg f)^m} \abs{\int_M d\eta \wedge (f^m)^*\tau}\\
		&\leq \frac{1}{(\deg f)^m} \norm{d\eta}_{n} \norm{(f^m)^*\tau}_{\frac{n}{n-1}}\\
		&\leq \frac{C}{(\deg f)^m} \norm{d\eta}_{n} \norm{(f^m)^*\tau}_{f, \frac{n}{n-1}}\\
		&= \frac{C}{(\deg f)^{\frac{m}{n}}} \norm{d\eta}_{n} \norm{\tau}_{f, \frac{n}{n-1}},
	\end{align*}
	where $C$ is the constant in \eqref{eq:norm_comparability}. Hence, the sequence $((\deg f)^{-m}(f^m)^*\omega)$ is Cauchy in the weak sense, and therefore has a weak limit $\mu_\omega$.
	
	Let now $\omega, \omega' \in L^{1, \sharp}(\wedge^n M; \C)$ with $\int_M \omega = \int_M \omega'$. Then $\omega$ and $\omega'$ belong in the same cohomology class $c \in \cehom{n}(M; \C)$ and $\omega - \omega' = d\tau$ for some $\tau \in \cesobt(\wedge^{n-1}M; \C)$. Now, by the same calculation as before, we have for every $\eta \in C^\infty(M)$ the estimate
	\begin{align*}
	\lim_{m \to \infty}\abs{\int_M \eta \left(
			\frac{(f^{m})^* \omega}{(\deg f)^{m}} - \frac{(f^{m})^* \omega'}{(\deg f)^{m}}\right)}
		&= \lim_{m \to \infty}\frac{1}{(\deg f)^m} \abs{\int_M \eta \, (f^m)^*d\tau}\\
		&\leq \lim_{m \to \infty}
			\frac{C}{(\deg f)^{\frac{m}{n}}} \norm{d\eta}_{n} \norm{\tau}_{f, \frac{n}{n-1}} 
		= 0.
	\end{align*}
	Hence, $\mu_\omega = \mu_{\omega'}$. The desired uniqueness condition now follows by linearity, since $(\int_M \omega)\omega'$ and $(\int_M \omega')\omega$ have the same integral over $M$ for all $\omega, \omega' \in L^{1, \sharp}(\wedge^n M; \C)$.
\end{proof}

We define the \emph{invariant measure $\mu_f$ of $f$} by $\mu_f = \mu_{\omega_0}$, where $\omega_0$ is the $n$-form $m_n(M)^{-1} \vol_M$, and $m_n$ is the Lebesgue measure on $M$. Under this notation, $\mu_\omega = (\int_M \omega)\mu_f$ for all $\omega \in L^{1, \sharp}(\wedge^k M; \C)$. Note that, since $f^*\vol_M = J_f \vol_M$ and quasiregular mappings have a positive Jacobian almost everywhere, the measure $\mu_f$ is a probability measure.

For technical reasons, we record the following variation of Theorem \ref{thm:okuyama_pankka_complex_n-forms}, which is obtained as an easy corollary of the proof of Theorem \ref{thm:okuyama_pankka_complex_n-forms}.

\begin{cor}\label{cor:okuyama_pankka_technical_variant}
	Let $f \colon M \to M$ be a non-constant uniformly quasiregular self-map on a closed $n$-manifold $M$ satisfying $\deg f \geq 2$. Then, for all $\omega \in L^{1, \sharp}(\wedge^n M; \C)$ satisfying $\int_M \omega = 0$ and all $\lambda \in \C$ with $\abs{\lambda} = \deg f$, we obtain
	\[
		\frac{(f^m)^* \omega}{\lambda^m} \xrightarrow[m \to \infty]{} 0
	\]
	in the weak sense.
\end{cor}
\begin{proof}
	Since $\int_M \omega = 0$, there exists $\tau \in \cesobt(\wedge^{n-1}M; \C)$ for which $\omega = d\tau$. Now, as in the proof of Theorem \ref{thm:okuyama_pankka_complex_n-forms}, we obtain
	\begin{align*}
		\lim_{m \to \infty}\abs{\int_M \eta\frac{(f^m)^* \omega}{\lambda^m}}
		&= \lim_{m \to \infty}\frac{1}{(\deg f)^m} \abs{\int_M \eta \, (f^m)^*d\tau}\\
		&\leq \lim_{m \to \infty}\frac{C}{(\deg f)^{\frac{m}{n}}} 
			\norm{d\eta}_{n} \norm{\tau}_{f, \frac{n}{n-1}}
		= 0,
	\end{align*}
	for all test forms $\eta \in C^\infty(M)$, where $C$ is again the constant in \eqref{eq:norm_comparability}. This yields the desired result.
\end{proof}

A key property of the measure $\mu_f$ is that its support is the Julia set of $f$.

\begin{thm}[{\cite[Theorem 1.2]{OkuyamaPankka2014paper}}]\label{thm:okuyama_pankka_properties}
	Let $f \colon M \to M$ be a non-constant uniformly quasiregular self-map on a closed $n$-manifold $M$ satisfying $\deg f \geq 2$, and let $\mu_f$ be the invariant measure of $f$. Then
	\[
		\spt \mu_f = \Julia_f.
	\]
\end{thm}

For the proof, see \cite[Section 6]{OkuyamaPankka2014paper}. This property is important, since it reduces the second claim of Theorem \ref{thm:positive_julia_set} that $m_n(\Julia_f) > 0$ to the first claim of Theorem \ref{thm:positive_julia_set} that $\mu_f$ is absolutely continuous with respect to the Lebesgue measure $m_n$.

%%%%%%%%%%%%%%%%%%%%%%%%%%%%%%%%%%%%%%%%%%%%%%%%%%%%%%%%%%%%%%%%%%%%%%%%%%%%%%%%%%%%%%%%%%%%%%%%%%
%%%%%%%%%%%%%%%%%%%%%%%%%%%%%%%%%%%%%%%%%%%%%%%%%%%%%%%%%%%%%%%%%%%%%%%%%%%%%%%%%%%%%%%%%%%%%%%%%%

\section{Proofs of the main results}\label{sect:main_results}

In this section, we prove Theorems \ref{thm:positive_julia_set} and \ref{thm:dynamical_bonk_heinonen}. Throughout this section $M$ is a closed $n$-manifold and $f \colon M \to M$ is a non-constant uniformly quasiregular map on $M$ with $\deg f \geq 2$.

We fix some terminology for the sake of presentation. Let $\lambda \in \C \setminus \{0\}$ and $0 \leq k \leq n$. A cohomology class $c \in \cehom{k}(M; \C) \setminus \{0\}$ is a \emph{$k$-eigenclass of $f$ with eigenvalue $\lambda$} if $f^*c = \lambda c$. Similarly, a differential form $\omega \in \Gamma(\wedge^k M; \C) \setminus \{0\}$ is a \emph{$k$-eigenform of $f$ with eigenvalue $\lambda$} if $f^*\omega = \lambda\omega$. The \emph{cohomological eigenspace of $f$ corresponding to the eigenvalue $\lambda$} is the complex vector subspace
\[
	E^k(f; \lambda) = \left\{c \in \cehom{k}(M; \C) : f^*c = \lambda c\right\}.
\]
of $\cehom{k}(M; \C)$. Finally, we say that two complex differential forms $\omega, \omega' \in \Gamma(\wedge^k M; \C)$ are \emph{complex $f$-orthogonal at a point $x \in M$} if $\ip{\omega_x}{\omega_x'}_f = 0$, and also that $\omega$ and $\omega'$ are \emph{complex $f$-orthogonal almost everywhere} if they are complex $f$-orthogonal at almost every point $x \in M$. 

\subsection{Proof of Theorem \ref{thm:positive_julia_set}} 

Recall the statement of Theorem \ref{thm:positive_julia_set}, which under our assumptions on $f$ and $M$ is as follows.

\begin{customthm}{\ref{thm:positive_julia_set}}
	If $M$ is not a rational homology sphere, then $\mu_f$ is absolutely continuous with respect to $m_n$. 
\end{customthm}

We obtain Theorem \ref{thm:positive_julia_set} from the following lemma.

\begin{lemma}\label{lemma:equilibrium_measure_integral_representation}
	Let $0 < k \leq n$, and suppose there exists a $k$-eigenform $\omega$ of $f$ satisfying $\omega \in L^{n/k, \sharp}(\wedge^k M; \C)$ and $\norm{\omega}_{n/k, f} = 1$. Then the invariant equilibrium measure $\mu_f$ of $f$ has a representation
	\[
		\mu_f = \abs{\omega}_f^\frac{n}{k} \vol_M
	\]
	as a Lebesgue measurable $n$-form. Consequently, $\mu_f$ is absolutely continuous with respect to the Lebesgue measure $m_n$ on $M$, and $\spt \omega = \spt \mu_f = \Julia_f$.
\end{lemma}
\begin{proof}
	Denote $\eta = \abs{\omega}_f^{n/k} \vol_M$. Note that
	\begin{equation}\label{eq:eigenvalue_abs_value}
		\abs{\lambda} = \norm{\lambda\omega}_{f, \frac{n}{k}} = \norm{f^*\omega}_{f, \frac{n}{k}} 
			= (\deg f)^\frac{k}{n}\norm{\omega}_{f, \frac{n}{k}} = (\deg f)^\frac{k}{n}.
	\end{equation}
	We now compute
	\begin{align*}
		f^*\eta 
		&= \left(\abs{\omega}_f^{\frac{n}{k}} \circ f \right) f^*\vol_M
		= \left(\abs{\omega}_f^{\frac{n}{k}} \circ f\right) J_f \vol_M
		= \abs{f^*\omega}_f^{\frac{n}{k}} \vol_M\\
		&= \abs{\lambda\omega}_f^{\frac{n}{k}} \vol_M
		= \abs{\lambda}^{\frac{n}{k}} \eta 
		= (\deg f)\eta.
	\end{align*}
	Hence, for all $m \geq 1$,
	\[
		\frac{(f^m)^* \eta}{(\deg f)^m} = \eta.
	\]
	Since $\omega \in L^{n/k, \sharp}(\wedge^k M; \C)$, we have $\eta \in L^{1, \sharp}(\wedge^n M; \C)$. Therefore, we may use Theorem \ref{thm:okuyama_pankka_complex_n-forms} on $\eta$, obtaining the desired result
	\[
		\mu_f = \lim_{m \to \infty} \frac{(f^m)^* \eta}{(\deg f)^m} = \eta.
	\]
\end{proof}

\begin{proof}[Proof of Theorem \ref{thm:positive_julia_set}]
	Assume that there exists $k \in \{1, \ldots, n-1\}$ for which $H^k(M; \R) \neq \{0\}$. Then by Lemma \ref{lemma:sobolev_de_rham_isomorphism}, $\cehom{k}(M) \neq \{0\}$. Since $\cehom{k}(M; \C) = \cehom{k}(M)\otimes \C$, we also have $\cehom{k}(M; \C) \neq \{0\}$. Hence, there exists a $k$-eigenclass $c$ of $f$.
	
	Let $\omega_c \in \fharm{k}(M; \C)$ be the complex $f$-harmonic $k$-form in the eigenclass $c$. By normalization, we may assume $\norm{\omega_c}_{n/k, f} = 1$. By Corollary \ref{cor:eigenvalues_translate_to_forms}, $\omega_c$ is a $k$-eigenform of $f$, and by Proposition \ref{prop:higher_integrability}, we have the higher integrability condition $\omega_c \in L^{n/k, \sharp}(\wedge^k M; \C)$. Thus, $\omega_c$ satisfies the assumptions of Lemma \ref{lemma:equilibrium_measure_integral_representation}, which proves the claim.
\end{proof}

We note that besides Theorem \ref{thm:positive_julia_set} the previous proof also yields Proposition \ref{prop:invariant_measure_representation}. Moreover, we obtain the following fact as an immediate corollary of Lemma \ref{lemma:equilibrium_measure_integral_representation}. 

\begin{cor}\label{cor:eigenvector_f-abs_coincidence}
	Let $0 < k \leq n$, and let $\omega_1$ and $\omega_2$ be $k$-eigenforms of $f$ satisfying $\omega_1, \omega_2 \in L^{n/k, \sharp}(\wedge^k M; \C)$ and $\norm{\omega_1}_{n/k, f} = \norm{\omega_2}_{n/k, f} = 1$. Then
	\[
		\abs{(\omega_1)_x}_{f} = \abs{(\omega_2)_x}_{f}
	\]
	for almost every $x \in M$.
\end{cor}
\begin{proof}
	The measurable $n$-forms $\eta_1 = \abs{\omega_1}_f^{n/k} \vol_M$ and $\eta_2 = \abs{\omega_2}_f^{n/k} \vol_M$ both represent the same measure $\mu_f$. Hence, $(\eta_1)_x = (\eta_2)_x$ for almost every $x \in M$, which yields the claim.
\end{proof}

\subsection{Proof of Theorem \ref{thm:dynamical_bonk_heinonen}}
Again, we recall the statement of Theorem \ref{thm:dynamical_bonk_heinonen} under our assumptions on $f$ and $M$.

\begin{customthm}{\ref{thm:dynamical_bonk_heinonen}}
	For all $k \in \{0, \ldots, n\}$,
	\[
		\dim H^k(M; \R) \leq \binom{n}{k}.
	\]
\end{customthm}

The key ingredient in the proof of Theorem \ref{thm:dynamical_bonk_heinonen} is the idea of Corollary \ref{cor:eigenvector_f-abs_coincidence} applied to inner products of eigenvectors. The proof divides into two cases, depending on whether two eigenvectors have the same eigenvalue or not. First, we show that higher integrable eigenforms of $f$ with different corresponding eigenvalues are $f$-orthogonal almost everywhere.

\begin{lemma}\label{lemma:eigenvector_orthogonality_different_eigenvalue}
	Let $0 < k < n$, and let $\omega_1$ and $\omega_2$ be $k$-eigenforms of $f$ with corresponding eigenvalues $\lambda_1$ and $\lambda_2$. Assume that $\omega_1, \omega_2 \in L^{n/k, \sharp}(\wedge^k M; \C)$,  $\norm{\omega_1}_{f, n/k} = \norm{\omega_2}_{f, n/k} = 1$, and $\lambda_1 \neq \lambda_2$. Then $\omega_1$ and $\omega_2$ are $f$-orthogonal almost everywhere.
\end{lemma}
\begin{proof}
	Let 
	\[
		\eta_{12} = \abs{\ip{\omega_1}{\omega_2}_f}^{\frac{n}{2k}-1} \ip{\omega_1}{\omega_2}_f \vol_M.
	\]
	Then 
	\begin{align*}
		f^* \eta_{12}
		&= \left(\left( \abs{\ip{\omega_1}{\omega_2}_f}^{\frac{n}{2k}-1} \ip{\omega_1}{\omega_2}_f \right) \circ f \right)
				J_f \vol_M\\
		&= \abs{\ip{f^*\omega_1}{f^*\omega_2}_f}^{\frac{n}{2k}-1} \ip{f^*\omega_1}{f^*\omega_2}_f \vol_M\\
		&= \lambda_1 \overline{\lambda_2} (\deg f)^\frac{n-2k}{n} \eta_{12}.
	\end{align*}
	We denote $\lambda_{12} = \lambda_1 \overline{\lambda_2} (\deg f)^{(n-2k)/n}$. By \eqref{eq:eigenvalue_abs_value}, $\abs{\lambda_1} = \abs{\lambda_2} = (\deg f)^{k/n}$. Hence, the assumption $\lambda_1 \neq \lambda_2$ implies $\lambda_{12} \neq \deg f$. Now, however,
	\[
		(\deg f)\int_M \eta_{12} = \int_M f^*\eta_{12} = \lambda_{12} \int_M \eta_{12},
	\]
	and consequently $\int_M \eta_{12} = 0$. Furthermore, since $\omega_1, \omega_2 \in L^{n/k, \sharp}(\wedge^k M; \C)$, we have $\eta_{12} \in L^{1, \sharp}(\wedge^n M; \C)$ by Hölder's inequality. Therefore, by Corollary \ref{cor:okuyama_pankka_technical_variant},
	\[
		\eta_{12} = \frac{(f^m)^*\eta_{12}}{(\lambda_{12})^{m}} \xrightarrow[m \to \infty]{} 0
	\]
	in the weak sense. Hence, $\eta_{12} = 0$, and consequently $\omega_1$ and $\omega_2$ are $f$-orthogonal almost everywhere.
\end{proof}

Next, we study higher integrable eigenforms of $f$ with the same corresponding eigenvalue.

\begin{lemma}\label{lemma:eigenvector_orthogonality_same_eigenvalue}
	Let $0 < k < n$, and let $\omega_1$ and $\omega_2$ be $k$-eigenforms of $f$ with the same corresponding eigenvalue $\lambda$. Assume that $\omega_1, \omega_2 \in L^{n/k, \sharp}(\wedge^k M; \C)$ and $\norm{\omega_1}_{f, n/k} = \norm{\omega_2}_{f, n/k} = 1$. Then there exists a constant $C_{12} \in \C$ satisfying
	\[
		\ip{(\omega_1)_x}{(\omega_2)_x}_f = C_{12} \abs{(\omega_1)_x}_f\abs{(\omega_2)_x}_f
	\]
	for almost every $x \in M$.
\end{lemma}
\begin{proof}
	Let 
	\[
		\eta_{12} = \abs{\ip{\omega_1}{\omega_2}_f}^{\frac{n}{2k}-1} \ip{\omega_1}{\omega_2}_f \vol_M.
	\]
	As previously, $\eta_{12} \in L^{1, \sharp}(\wedge^n M; \C)$ and $f^* \eta_{12} = \lambda_{12}\eta_{12}$, where this time $\lambda_{12} = \lambda \overline{\lambda} (\deg f)^{(n-2k)/n} = \deg f$. By Theorem \ref{thm:okuyama_pankka_complex_n-forms},
	\[
		\eta_{12} = A\mu_f
	\]
	for some $A \in \C$. 
	
	Furthermore, let 
	\[
		\sigma_{12} = \abs{\ip{\omega_1}{\omega_2}_f}^{\frac{n}{2k}} \vol_M
	\]
	Then we have $\sigma_{12} \in L^{1, \sharp}(\wedge^n M; \C)$ and $f^* \sigma_{12} = (\deg f) \sigma_{12}$. Thus, by Theorem \ref{thm:okuyama_pankka_complex_n-forms},
	\[
		\sigma_{12} = B\mu_f = B\abs{\omega_1}^\frac{n}{k}_f \vol_M = B\abs{\omega_2}^\frac{n}{k}_f \vol_M
	\]
	for some $B \in [0, \infty)$. 
	
	If $B = 0$, then $\sigma_{12} = 0$, in which case $\ip{(\omega_1)_x}{(\omega_2)_x}_f = 0$ at almost every $x \in M$. In this case, we may choose $C_{12} = 0$. Suppose now that $B \neq 0$. Then
	\[
		\ip{(\omega_1)_x}{(\omega_2)_x}_f 
		= AB^{-1}\abs{\ip{(\omega_1)_x}{(\omega_2)_x}_f} 
		= A B^{\frac{2k}{n} - 1} \abs{(\omega_1)_x}_f \abs{(\omega_2)_x}_f,
	\]
	for almost every $x \in M$, which proves the claim for $C_{12} = AB^{2k/n - 1}$.
\end{proof}

Note that Proposition \ref{prop:norm_mimimizer_angles} is a direct consequence of combining Lemmas \ref{lemma:eigenvector_orthogonality_different_eigenvalue} and \ref{lemma:eigenvector_orthogonality_same_eigenvalue} with Proposition \ref{prop:higher_integrability} and Corollary \ref{cor:eigenvalues_translate_to_forms}. Furthermore, by a standard Gram--Schmidt argument, Lemma \ref{lemma:eigenvector_orthogonality_same_eigenvalue} yields the following corollary. 
\begin{cor}\label{cor:orthogonal_set_same_eigenvalue}
	Let $k \in \{1, \ldots, n-1\}$, $\lambda \in \C \setminus \{0\}$, and let $\omega_1, \ldots, \omega_l$ be linearly independent $k$-eigenforms of $f$ with corresponding eigenvalue $\lambda$. Assume that $\norm{\omega_i}_{f, n/k} = 1$ and $\omega_i \in L^{n/k, \sharp}(\wedge^k M; \C)$ for all $i \in \{1, \ldots, l\}$. 
	Then there exist $k$-eigenforms $\tau_1, \ldots, \tau_l$ of $f$ with corresponding eigenvalue $\lambda$, which are pairwise $f$-orthogonal almost everywhere and satisfy $\norm{\tau_i}_{f, n/k} = 1$ and $\tau_i \in L^{n/k, \sharp}(\wedge^k M; \C)$ for all $i \in \{1, \ldots, l\}$.
\end{cor}

We are now ready to prove Theorem \ref{thm:dynamical_bonk_heinonen}.

\begin{proof}[Proof of Theorem \ref{thm:dynamical_bonk_heinonen}]
	We may assume that $H^k(M; \R) \neq 0$. Since $H^0(M; \R)$ and $H^n(M; \R)$ are one dimensional, we may also assume that $0 < k < n$.  By Lemma \ref{lemma:sobolev_de_rham_isomorphism} and the fact that $\cehom{k}(M; \C) = \cehom{k}(M) \otimes \C$, it is enough to show that
	\[	
		\dim_\C \cehom{k}(M; \C) \leq \binom{n}{k}.
	\] 
	We denote $D = \dim_\C \cehom{k}(M; \C)$.
	
	Let $\lambda_1, \ldots, \lambda_l$ be the eigenvalues of $f^* \colon \cehom{k}(M; \C) \to \cehom{k}(M; \C)$, and denote $E_i = E^k(f; \lambda_i)$ for $i \in \{1, \ldots, l\}$. For every eigenspace $E_i$, we fix a free basis $c_{i, 1}, \ldots, c_{i, l_i}$, and let  $\omega_{i, 1}, \ldots, \omega_{i, l_i}$ be the corresponding $f$-harmonic $k$-forms. 
	
	By Corollary \ref{cor:eigenvalues_translate_to_forms}, the forms $\omega_{i, j}$ are $k$-eigenforms of $f$ with corresponding eigenvalues $\lambda_i$, and by normalizing, we may assume that $\smallnorm{\omega_{i, j}}_{n/k, f} = 1$ for all $i \in \{1, \ldots, l\}$ and $j \in \{1, \ldots, l_i\}$. Furthermore, by Proposition \ref{prop:higher_integrability}, $\omega_{i, j} \in L^{n/k, \sharp}(\wedge^k M; \C)$ for all $i \in \{1, \ldots, l\}$ and $j \in \{1, \ldots, l_i\}$.
	
	We may now apply Corollary \ref{cor:orthogonal_set_same_eigenvalue}, obtaining for every $i \in \{1, \ldots, l\}$ $k$-eigenforms $\tau_{i, j}$ of $f$ which are pairwise $f$-orthogonal almost everywhere and satisfy $\smallnorm{\tau_{i, j}}_{n/k, f} = 1$ and $\tau_i \in L^{n/k, \sharp}(\wedge^k M; \C)$ for every $j \in \{1, \ldots, l_i\}$. By Lemma \ref{lemma:eigenvector_orthogonality_different_eigenvalue}, we also have that $\tau_{i_1, j_1}$ and $\tau_{i_2, j_2}$ are $f$-orthogonal almost everywhere for all $i_1, i_2 \in \{1, \ldots, l\}$ satisfying $i_1 \neq i_2$ and every $j_s \in \{1, \ldots l_s \}$ for $s = 1, 2$.
	
	We fix Borel representatives $\tau_{i, j}^{\text{Bor}}$ of the measurable $k$-forms $\tau_{i, j}$, and denote
	\[
		B = \left\{ \tau_{i, j}^{\text{Bor}} \colon i \in \{1, \ldots, l\}, j \in \{1, \ldots, l_i\} \right\}.
	\]
	Now, if $\tau, \tau' \in B$ satisfy $\tau \neq \tau'$, then $\tau$ and $\tau'$ are $f$-orthogonal almost everywhere. Furthermore, $\abs{B} = D$, and by Lemma \ref{lemma:equilibrium_measure_integral_representation}, $\spt \tau = \Julia_f$ for every $\tau \in B$.
	
	Since $H^k(M; \R) \neq 0$, $m_n(\Julia_f) > 0$ by Theorem \ref{thm:positive_julia_set}. Therefore, we may fix $x \in \Julia_f$ for which the inner product $\ip{\cdot}{\cdot}_f$ is defined at $x$, $\abs{\tau_x}_f > 0$ for all $\tau \in B$, and $\ip{\tau_x}{\tau'_x}_f = 0$ for all $\tau, \tau' \in B$ satisfying $\tau \neq \tau'$. Hence, the set $B_x = \{\tau_x \colon \tau \in B\}$ is a linearly independent subset of $(\wedge^k T^*_x M) \otimes \C$, and $\abs{B_x} = D$. We conclude that 
	\[
		D \leq \dim_\C \left((\wedge^k T^*_x M) \otimes \C\right) = \binom{n}{k},
	\]
	which proves Theorem \ref{thm:dynamical_bonk_heinonen}.
\end{proof}

%%%%%%%%%%%%%%%%%%%%%%%%%%%%%%%%%%%%%%%%%%%%%%%%%%%%%%%%%%%%%%%%%%%%%%%%%%%%%%%%%%%%%%%%%%%%%%%%%%
%%%%%%%%%%%%%%%%%%%%%%%%%%%%%%%%%%%%%%%%%%%%%%%%%%%%%%%%%%%%%%%%%%%%%%%%%%%%%%%%%%%%%%%%%%%%%%%%%%

%%%%%%%%%%%%%%% Bibliography %%%%%%%%%%%%%%%

\end{document}